\documentclass[11pt, a4paper]{article}
\usepackage[margin=1in]{geometry}

\usepackage{amsmath}
\usepackage{amsthm}
\usepackage{amsfonts}
\usepackage[utf8]{inputenc}
\usepackage{xcolor}
\usepackage{amssymb}
\usepackage[hidelinks]{hyperref}
\usepackage{fullpage}
\usepackage{esvect,mathtools}
\usepackage[numbers,square,sort]{natbib}
\usepackage[capitalize]{cleveref}
\usepackage{subfiles}
\usepackage{graphicx}
\usepackage{subcaption}
\usepackage{float}
\graphicspath{ {./figures/} }

\newtheorem{theorem}{Theorem}
\numberwithin{theorem}{section} 
\newtheorem{lemma}[theorem]{Lemma}
\newtheorem{corollary}[theorem]{Corollary}
\newtheorem{question}[theorem]{Question}
\newtheorem{conjecture}[theorem]{Conjecture}
\theoremstyle{remark}
\newtheorem{remark}[theorem]{Remark}
\newtheorem{observation}[theorem]{Observation}

\theoremstyle{definition}
\newtheorem{definition}[theorem]{Definition}

\newcommand\ab[1]{\lvert #1\rvert}

\title{Ramsey numbers of digraphs with local edge structure}
\author{Domagoj Bradač\thanks{Department of Mathematics, EPFL, Lausanne, Switzerland. Email: \texttt{domagoj.bradac@epfl.ch}. Research supported in part by SNSF grant 200021-228014. } \and Patryk Morawski\thanks{Department of Mathematics, ETH Z\"urich, Z\"urich, Switzerland.
Email: \texttt{patryk.morawski@ifor.math. ethz.ch}. Research supported in part by SNSF grant 200021-228014.} \and Benny Sudakov\thanks{Department of Mathematics, ETH Z\"urich, Z\"urich, Switzerland. Email: \texttt{benjamin.sudakov@math.ethz.ch}. Research supported in part by SNSF grant 200021-228014.} \and Yuval Wigderson\thanks{Institute for Theoretical Studies, ETH Z\"urich, Z\"urich, Switzerland. 
    Email: {\texttt{yuval.wigderson@eth-its. ethz.ch}}. Research supported by Dr.\ Max R\"{o}ssler, the Walter Haefner Foundation, and the ETH Z\"{u}rich Foundation.}}
\date{}

\begin{document}

\maketitle
\begin{abstract}
One of the classical topics in graph Ramsey theory is the study of which $n$-vertex graphs have Ramsey numbers that are linear in $n$.
In this paper, we consider this problem in the context of directed graphs. The oriented Ramsey number of a digraph $G$ is the smallest integer $N$ such that every $N$-vertex tournament contains a copy of $G$. We prove that every bounded-degree acyclic digraph with a ``local edge structure'' has a linear oriented Ramsey number.

More precisely, we say that a digraph $G$ has graded bandwidth $w$ if its vertices can be partitioned into sets $V_1, \dots, V_H$ such that all edges $uv \in E(G)$ with $u \in V_i$ and $v \in V_j$ satisfy
$1 \leq j - i \leq w$. 
We prove that $\vv{r}(G) \leq 3^{57\Delta w} |V(G)|$ for any acyclic $G$ with graded bandwidth $w$ and maximum degree $\Delta$. 

This provides a common generalization of several prior results, including on digraphs of bounded height, of digraphs of bounded bandwidth, and blowups of bounded-degree oriented trees. This notion also captures a wide variety of natural digraphs, such as oriented grids and hypercubes.
\end{abstract}

\section{Introduction}
\subsection{Background}
The Ramsey number $r(G)$ of a graph $G$ is the minimum number $N$ such that every 2-coloring of the edges of the complete graph $K_N$ on $N$ vertices contains a monochromatic copy of $G$.
Studying how $r(G)$ depends on $G$ is the main question in graph Ramsey theory, and is one of the most studied topics in combinatorics. In particular, one would like to obtain comparable upper and lower bounds on $r(G)$ in terms of basic parameters of $G$.
Such bounds are known in certain regimes; for example, it is known that if $G$ has $n$ vertices, then $r(G)$ grows exponentially in $n$ if and only if $G$ has $\Omega(n^2)$ edges \cite{MR3065018,MR2782204}.

In particular, sparse graphs---those with $o(n^2)$ edges---have subexponential Ramsey numbers. However, it was observed very early in the history of graph Ramsey theory (see e.g.\ \cite{MR373959}) that certain sparse graphs, such as trees and cycles, have \emph{much} smaller Ramsey numbers, namely \emph{linear} in their order.
In 1975, Burr and Erd\H{o}s \cite{MR371701} conjectured that this is in fact true for all sparse graphs.
Their notion of sparsity is the degeneracy\footnote{In fact, they originally stated their conjecture in terms of the \emph{arboricity} of $G$, but it is easy to verify that a graph has bounded degeneracy if and only if it has bounded arboricity. Most subsequent papers on this topic use degeneracy, which seems more well-suited for Ramsey-theoretic study.} of the graph, where a graph $G$ is called $d$-degenerate if every subgraph of $G$ has a vertex of degree at most $d$, and the degeneracy of $G$ is the least $d$ for which it is $d$-degenerate.
The Burr--Erd\H{o}s conjecture then states that for every $d$-degenerate graph $G$ we have that $r(G) \leq C_d |V(G)|$, where $C_d$ is a constant that depends only on $d$.
The first major step towards proving this conjecture was made by Chv\'{a}tal, R\"odl, Szemer\'{e}di and Trotter \cite{MR714447}, who proved that that for every graph $G$ with maximum degree $\Delta$, we have $r(G) \leq C_\Delta |V(G)|$, that is, that the conjecture holds under the stronger assumption of bounded maximum degree.
After a sequence of further partial results (e.g.\ \cite{MR2548655,MR2037075,MR1834850,MR1198403,MR1788033,1501.05350}), the full Burr--Erd\H{o}s conjecture was finally resolved by Lee \cite{MR3664811} in 2017.

For bounded-degree graphs, we have a fairly precise understanding of how large their Ramsey numbers are. Substantially improving on the early result of Chvat\'al, R\"odl, Szemer\'edi, and Trotter \cite{MR714447}, Graham, R\"odl, and Ruci\'nski \cite{MR1788033} proved that every $n$-vertex graph $G$ with maximum degree $\Delta$ satisfies $r(G) \leq C^{\Delta (\log \Delta)^2}n$, where $C$ is an absolute constant. In a subsequent paper \cite{MR1832445}, they noted that their technique yields a better bound of $r(G) \leq C^{\Delta \log \Delta}$ in case $G$ is bipartite, and also showed that this bound is close to best possible, in that there exist bipartite $n$-vertex graphs maximum degree $\Delta$ satisfying $r(G) \geq c^{\Delta} n$, for another absolute constant $c>1$. Subsequent work by Conlon \cite{MR2550376}, Fox--Sudakov \cite{MR2520279}, and Conlon--Fox--Sudakov \cite{MR3004807} removed one logarithmic factor from both upper bounds; in particular, it is now known that every $n$-vertex bipartite graph with maximum degree $\Delta$ satisfies $r(G) \leq C^{\Delta}n$, which is best possible up to the value of $C$.

We now turn our attention to directed graphs (\emph{digraphs} for short), where similar questions can be asked, and which are the main focus of this paper. A digraph $G$ is \emph{acyclic} if it contains no directed cycles. 
For an acyclic $G$, we define the \emph{oriented Ramsey number} $\vv{r}(G)$ as the minimum number $N$ such that every tournament on $N$ vertices, that is, every edge-orientation of the complete graph $K_N$, contains a copy of $G$.
The study of oriented Ramsey numbers was initiated in 1951 by Stearns \cite{MR109087}, who showed that for a transitive tournament $\vv{T_n}$ on $n$ vertices we have $\vv{r}(\vv{T_n}) \leq 2^{n-1}$, which was complemented by a lower bound of $\vv{r}(\vv{T_n}) \geq 2^{\frac{n}{2} - 1}$ by Erd\H{o}s and Moser \cite{MR168494} in 1964.
As in the undirected setting, for sparser digraphs $G$ this number is in general much smaller than exponential in $|V(G)|$; for example, for $n > 8$ every orientation of the $n$-vertex path has oriented Ramsey number equal to $n$ \cite{MR837805,MR1750898}.
For more general oriented trees, Sumner conjectured that $\vv{r}(T) \leq 2n - 2$ for any oriented tree $T$ on $n$ vertices.
Sumner's conjecture has attracted a great deal of interest over the years (e.g.\ \cite{MR837805,MR4278113,MR2078502,MR1136161,MR1874730,MR1791347,MR4392271,MR4809293,MR2832810}) and in 2011 it was proved for sufficiently large $n$ by K\"uhn, Mycroft, and Osthus \cite{MR2793448}.

Motivated by these results, Buci\'c, Letzter and Sudakov \cite{MR3980089} asked whether a natural anologue of the Burr--Erd\H{o}s conjecture holds for acyclic digraphs, that is, whether for all acyclic digraphs $G$ with maximum degree\footnote{By the maximum degree of a digraph, we mean the maximum degree of the underlying undirected graph.} $\Delta$ we have that $\vv{r}(G) \leq c_\Delta |V(G)|$ for some constant $c_\Delta$ depending only on $\Delta$.
Quite surprisingly, Fox, He, and Wigderson \cite{MR4819947} recently answered this question in the negative by showing that for any $\Delta$ and large enough $n$ there exists an $n$-vertex digraph $G$ with maximum degree $\Delta$ and $\vv{r}(G) \geq n^{\Omega(\Delta^{2/3} / \log^{5/3} \Delta)}$. 
In the other direction, they proved an upper bound of $\vv{r}(G) \leq n^{C_\Delta \log n}$ for any acyclic digraph $G$ with maximum degree $\Delta$. However, their results leave open the question of whether the worst case behavior for fixed $\Delta$ is always polynomial, or whether it can indeed be super-polynomial.

However, the motivation of Buci\'c--Letzter--Sudakov \cite{MR3980089} is sensible, so a very natural question now arises: why is it that many examples of bounded-degree digraphs \emph{do} have linear Ramsey number, now that we know that some bounded-degree digraphs do not? While we are very far from having a complete explanation, it appears that the answer to this question is controlled by whether $G$ has a ``simple'' structure. For example, in addition to the results for trees and cycles discussed above, Aboulker et al.\ \cite{2410.23566} recently proved that constant-sized \emph{blowups}\footnote{A \emph{blowup} of a digraph $G$ is obtained by replacing each vertex by an independent set, and each oriented edge by a complete bipartite graph all of whose edges are oriented the same way.} of oriented trees have linear Ramsey number, and another result of Fox--He--Wigderson \cite{MR4819947} is that bounded-degree digraphs of bounded \emph{height}\footnote{The \emph{height} of a digraph $G$ is the length of its longest directed path. Equivalently, this is the least $h$ such that there is a partition $V(G)=V_1 \cup \dotsb \cup V_h$ with the property that all edges are directed from $V_i$ to $V_j$ for some $i<j$.} have linear Ramsey number. In both instances, the structural assumption on $G$ is crucial to the proof, as it demonstrates that the edges of $G$ cannot be arbitrarily badly distributed. At the other extreme, the lower bound construction of Fox--He--Wigderson \cite{MR4819947} is what they term an \emph{interval mesh}; loosely speaking, this is a digraph whose edges are uniformly spread out at all scales, and which in particular has no local structure.
\subsection{Graded digraphs}

Another important class of digraphs studied by Fox--He--Wigderson \cite{MR4819947} is the class of \emph{graded} digraphs, which we now define.
\begin{definition}
    We say that a digraph $G$ is \emph{graded} with a \emph{graded partition} $V(G) = V_1 \cup \dots \cup V_H$ if every edge of $G$ is directed from $V_i$ to $V_{i+1}$ for some $i \in [H-1]$. 
\end{definition}
We remark that graded digraphs are necessarily acyclic, and that the graded partition is unique assuming that the underlying graph of $G$ is connected. In particular, the number $H$ of parts in the graded partition is equal to the height of $G$. 

There are many natural examples of graded digraphs. For example, the $d$-dimensional \emph{grid digraph} $\vv{\Gamma_{d,k}}$ whose vertex set is $[k]^d$ and whose edges are all ordered pairs of the form $$((x_1,\dots,x_i,\dots,x_d), (x_1,\dots,x_i+1,\dots,x_d)) \qquad \text{ for some }i \in [d]$$ is graded; one obtains the graded partition by setting $V_i \coloneqq \{(x_1,\dots,x_d) \in [k]^d : x_1+\dotsb+x_d=i\}$. An important special case of this construction is the \emph{oriented hypercube} $\vv{Q_d}$, which is obtained from the unoriented hypercube graph by directing all edges away towards the positive orthant. More generally, the Hasse diagram of any graded poset is a graded digraph.

Fox, He, and Wigderson \cite[Theorem 1.5]{MR4819947} 
proved that bounded-degree graded digraphs have Ramsey numbers that are at most polynomial in their order, namely that if $G$ is an $n$-vertex graded digraph with maximum degree $\Delta$, then $\vv r(G) \leq n^{11\Delta \log \Delta}$. Our first main result improves this polynomial bound to a linear bound, thus extending the set of structural assumptions which imply a positive answer to the question of Buci\'c--Letzter--Sudakov \cite{MR3980089}.
\begin{theorem}\label{theorem:easy_upper_bound}
    If $D$ is a graded digraph on $n$ vertices with maximum degree $\Delta$, then
    \[
        \vv r(D) \leq 10^9 \Delta^3 2^{4\Delta}n.
    \]
    More precisely, if $D$ has maximum in-degree $\Delta^-$ and maximum out-degree $\Delta^+$, then
    $$\vv{r}(D) \leq 10^9 \Delta^+ (\Delta^-)^2 2^{4\Delta^-} n.$$
\end{theorem}
We stress that the height of $D$ does not affect the bound in Theorem \ref{theorem:easy_upper_bound} at all. This is somewhat surprising, given the intuition above that the height of a digraph should play an analogous role to the chromatic number of a graph, and should in turn affect the oriented Ramsey number. Nonetheless, an understanding that arises from our techniques is that graded digraphs ``behave like'' bipartite graphs, regardless of their height. 

As an immediate corollary of Theorem \ref{theorem:easy_upper_bound}, we obtain a linear upper bound on the oriented Ramsey numbers of grid digraphs in any fixed dimension, since $\vv{\Gamma_{d,k}}$ has maximum in- and out-degree equal to $d$.
\begin{corollary}
    For any $d \geq 1$, there exists a constant $C_d = 10^9 d^3 2^{4d}$ such that the $d$-dimensional grid digraph $\vv{\Gamma_{d,k}}$  satisfies $\vv r(\vv{\Gamma_{d,k}}) \leq C_d \ab{V(\vv{\Gamma_{d,k}})}$.
\end{corollary}
We remark that there has recently been a great deal of interest in Ramsey- and Tur\'an-type questions involving grid graphs, see e.g.\  \cite{conlon2023size, clemens2021size, gao2023extremal,bradac2023turan,gishboliner2022constructing,furedi2013uniform,kim2016two,mota2015ramsey}. In particular, it is proved in \cite[Corollary 1.4]{mota2015ramsey} that the \emph{undirected} Ramsey number of the two-dimensional $k\times k$ grid graph is $(\frac 32+o(1))k^2$.

At the other extreme, where $k$ is fixed and $d$ tends to infinity, we obtain a polynomial bound. For example, for the oriented hypercube $\vv{Q_d}$, Theorem \ref{theorem:easy_upper_bound} implies that $\vv r(\vv{Q_d})\leq 2^{5d+o(d)} = \ab{V(\vv{Q_d})}^{5+o(1)}$. By optimizing our techniques, we are able to improve the exponent from $5$ to $\log_2(17) \approx 4.09$.
\begin{theorem}\label{thm:hypercube}
    There exists an absolute constant $C>0$ such that $\vv{r}(\vv{Q_d}) \leq C d^3 17^{d}$.
\end{theorem}
In fact, motivated by the example of the hypercube, in  \cref{section:appendix} we prove a strengthening of \cref{theorem:easy_upper_bound}, which gives a better bound for graded digraphs where the large in-degrees are only in the parts of the graded partition where the number of vertices is small. Such a result is useful for $\vv{Q_d}$, since in the hypercube, almost all vertices (and in particular those vertices lying in the very large parts of the graded partition) have in- and out-degree close to $d/2$.

In the undirected setting, it is a major open problem to determine $r(Q_d)$.
A famous conjecture of Burr and Erd\H{o}s \cite{MR371701} from 1975 is that the Ramsey number of the hypercube is linear in its order, i.e.\ that $r(Q_d) = O(2^d)$. This question has been intensively studied (see e.g.\ \cite{MR787944,MR1832445,MR1848785,MR2285201,MR2520279,MR3548291}); the current best known bound is due to Tikhomirov \cite{MR4722306}, who proved that $r(Q_d) \leq 2^{(2 - \varepsilon)d}$, where $\varepsilon > 0$ is some small absolute constant.
However, the results in the undirected setting cannot be used directly to obtain upper bounds on $\vv r(\vv{Q_d})$, and to the best of our knowledge \cref{thm:hypercube} is the first known polynomial bound on $\vv r(\vv{Q_d})$. Instead, the techniques from the undrected setting naturally yield polynomial upper bounds on the oriented Ramsey number of the \emph{bipartite orientation} of the hypercube, where all edges are directed from vertices of even to odd Hamming weight, rather than the more natural orientation $\vv{Q_d}$.

Our second main result shows that \cref{theorem:easy_upper_bound} is close to best possible, in the sense that there exist graded digraphs with maximum degree $\Delta$ and oriented Ramsey number of at least $c^\Delta |V(G)|$ for some absolute constant $c>1$. A very similar result was proved in the undirected setting by Graham, R\"odl and Ruci\'nski \cite{MR1832445}; in fact, it is not hard to adapt their construction to show the existence of such a $G$ which is \emph{bipartite}, that is, of height $2$. However, by modifying their construction appropriately, we are able to prove such a result for $G$ of arbitrary height, with layers of equal size, where the lower bound again does not depend on the height.

\begin{theorem}\label{theorem:lower_bound}
    There exist constants $c > 1$ and $\Delta_0$ such that for all $\Delta \geq \Delta_0$, $n \geq \Delta$, and $H \geq 2$ there exists a graded digraph $G$ with maximum degree $\Delta$ and with a graded partition $V(G) = V_1 \cup \dots \cup V_H$ such that $|V_i| \leq n$ for all $i \in [H]$, such that $\vv{r}(G) > c^\Delta H n \geq c^\Delta \ab{V(G)}$.
\end{theorem}

\subsection{Digraphs with bounded graded bandwidth}
Our proof of \cref{theorem:lower_bound}, like the proof of 
Graham, Rödl and Ruciński \cite{MR1832445} 
in the undirected case,
exploits the strong expansion properties of random bipartite graphs, showing that such graphs yield the desired lower bound. However, it is now well understood that graphs without such expansion properties have smaller Ramsey numbers. A key notion capturing this lack of expansion is \emph{bandwidth}.

We say that a graph $G$ has \emph{bandwidth} at most $w$ if there exists a labeling of its vertices with $[n]$ such that for every edge $\{i, j\} \in E(G)$ we have $|i - j| \leq w$.
Allen, Brightwell and Skokan \cite{MR3071850} showed that for any $\Delta$ there exists a constant $\beta$ such that if $G$ has maximum degree $\Delta$ and bandwidth at most $\beta|V(G)|$ then $r(G) \leq (2\Delta + 6)|V(G)|$.
Proving the Burr--Erd\H{o}s conjecture for graphs with bounded bandwidth was a crucial step towards its full resolution in 2017 by Lee \cite{MR3664811}; prior to this, he showed \cite{1501.05350} that for any $d$ and $\epsilon$, if $G$ is a sufficiently large graph with degeneracy $d$ and bandwidth at most $n^{1 - \epsilon}$, then $r(G) \leq (2d + 7)|V(G)|$.
Bandwidth has also emerged as a fundamental parameter in various other areas of graph theory, such as the famous bandwidth theorem of B\"ottcher, Schacht, and Taraz \cite{MR2448444}.

One can define the bandwidth of an acyclic digraph in the same way, now requiring that every edge is directed from $i$ to $j$, and that these satisfy $1 \leq j-i \leq w$. This is another natural structural notion, so it is not too surprising given the discussion above that digraphs of bounded bandwidth have linear Ramsey numbers. Indeed, answering a question of Yuster \cite{MR4156082}, it was proved by Dragani\'c et al.\ \cite{MR4328353} that every $n$-vertex acyclic digraph $G$ with bandwidth $w$ satisfies $\vv r(G) \leq 2^{4w+6}n$.

For a digraph, having bounded bandwidth and being graded are two structural notions which capture the idea of having ``local edge structure''. There is a natural common generalization of them, which we now define.

\begin{definition}[Graded bandwidth]
    We say that a digraph $G$ has \emph{graded bandwidth} at most $w$ if its vertex set can be partitioned into $V_1, \dots, V_H$ such that for every $uv \in E(G)$ with $u \in V_i$ and $v \in V_j$ we have $1 \leq j -i \leq w$.
    We call the sets $V_1, \dots, V_H$ the \emph{layers} of the graded partition of $G$.
\end{definition}
\noindent Note that in case $w=1$, this precisely recovers the definition of a graded digraph. 

It follows from \cite[Theorem 3.12]{MR4819947} that every $n$-vertex digraph $G$ with maximum degree $\Delta$ and graded bandwidth $w$ has $\vv{r}(G) \leq n^{C_{\Delta,w}}$ for some constant $C_{\Delta,w} = O(\Delta(\log \Delta+\log w))$.
Our next main result substantially improves this polynomial bound, stating that all bounded-degree digraphs with bounded graded bandwidth have linear oriented Ramsey numbers.

\begin{theorem}\label{theorem:main_theorem}
    Let $\Delta, w \in \mathbb{N}$ and let $G$ be a digraph with maximum degree $\Delta$ and graded bandwidth $w$.
    Then
    \[
        \vv{r}(G) \leq 3^{57\Delta w} |V(G)|.
    \]
\end{theorem}
This theorem is perhaps gives the most general known conditions which imply that a digraph has a linear Ramsey number. Indeed, it recovers the result of Dragani\'c et al.\ \cite{MR4328353} on digraphs of bounded bandwidth (as such digraphs trivially have bounded graded bandwidth), as well as the result of Fox--He--Wigderson \cite{MR4819947} on digraphs of bounded height (as a digraph of height $h$ has graded bandwidth at most $h$). It also implies a special case of the theorem of Aboulker et al.\ \cite{2410.23566} on blowups of oriented trees, as a constant-sized blowup of a bounded-degree oriented tree also has bounded degree and bounded graded bandwidth. Finally, of course, \cref{theorem:main_theorem} recovers the linear bound on graded digraphs from \cref{theorem:easy_upper_bound}. We remark that in all of these cases, the quantitative dependencies are worse than those arising from the original proofs; in particular, if one sets $w=1$ in \cref{theorem:main_theorem} one obtains a weaker statement than \cref{theorem:easy_upper_bound}; unsurprisingly, proving a more general result entails obtaining weaker bounds.


\vspace{5px}

The remainder of the paper is structured as follows. 
We first prove \cref{theorem:main_theorem} in \cref{section:main_theorem_proof}, which begins with a detailed proof outline.
In \cref{section:graded_upper_bound}, we then state a strengthening of \cref{theorem:easy_upper_bound} and optimizations in its proof compared to \cref{theorem:main_theorem}.
The formal proof is deferred to \cref{section:appendix}.
In \cref{section:lower_bound}, we first give a sketch of the proof of \cref{theorem:lower_bound}, then prove it for height-$2$ digraphs, and finally prove it in full generality.
We end in \cref{section:concluding_remarks} with some concluding remarks and open problems.

\paragraph{Notation:} For a directed graph $D$, we use $V(D)$ to denote its vertex set and $E(D)$ to denote its edge set, which is a collection of ordered pairs of elements of $V(D)$.
For a vertex $v \in V(D)$ we write $N_D^+(v)$ and $N_D^-(v)$ for its out- and in-neighborhood and $d_D^+(v)$ and $d_D^-(v)$ for its out- and in-degree, respectively.
For a subset $U \subseteq V(D)$ we let $N_D^+(U)$ denote the common out-neighborhood of $U$ and similarly $N_D^-(U)$ the common in-neighborhood of $U$.
Additionally, for a vertex $v \in V(D)$ we let $d_D^+(v, U) = |N_D^+(v) \cap U|$ and $d_D^-(v, U) = |N_D^-(v) \cap U|$ be the out- and in-degree of $v$ into $U$. For two sets $A, B \subseteq V(D),$ we write $e_D(A, B) = |\{ (a, b) \in E(D)\, \vert \, a \in A, b\in B\}|$ for the number of edges between $A$ and $B$, and we write $d_D(A, B) = \frac{e_D(A, B)}{|A||B|}$, for the density between $A$ and $B$.
Whenever the digraph $D$ is clear from context, we omit the subscript and write $d^+(v)$ for $d_D^+(v)$, etc.
Throughout the paper, we omit ceilings and floors whenever they are not crucial.

\section{Upper bound for the general case}\label{section:main_theorem_proof}
\subsection{Proof Outline}\label{section:proof_outline}
\begin{figure}[H]
    \centering
    \begin{subfigure}[h]{0.45\linewidth}
        \includegraphics[width=\linewidth]{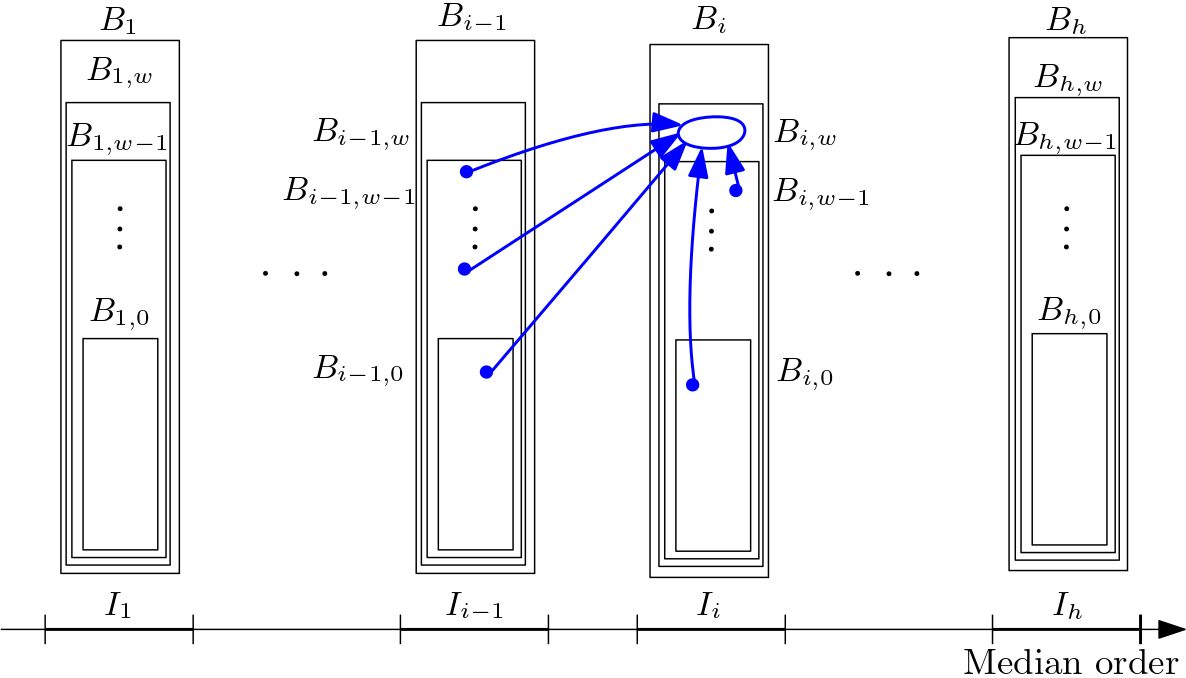}
        \caption{In \cref{theorem:embedding_structure}, we find sets $B_1, \dots, B_h \subseteq V(T)$, together with subsets $B_{i, 0} \subseteq B_{i, 1} \subseteq \dots \subseteq B_{i, w} \subseteq B_i$ for each $i \in [h]$, such that for any $i \in [h]$ and $j \in [w]$ almost all $\Delta$-subsets of $B_{i-1} \cup B_{i, j-1}$ have a large common out-neighborhood in $B_{i, j}$.}
        \label{subfig:embedding_structure}
    \end{subfigure}
    \hfill
    \begin{subfigure}[]{0.45\linewidth}
        \includegraphics[width=\linewidth]{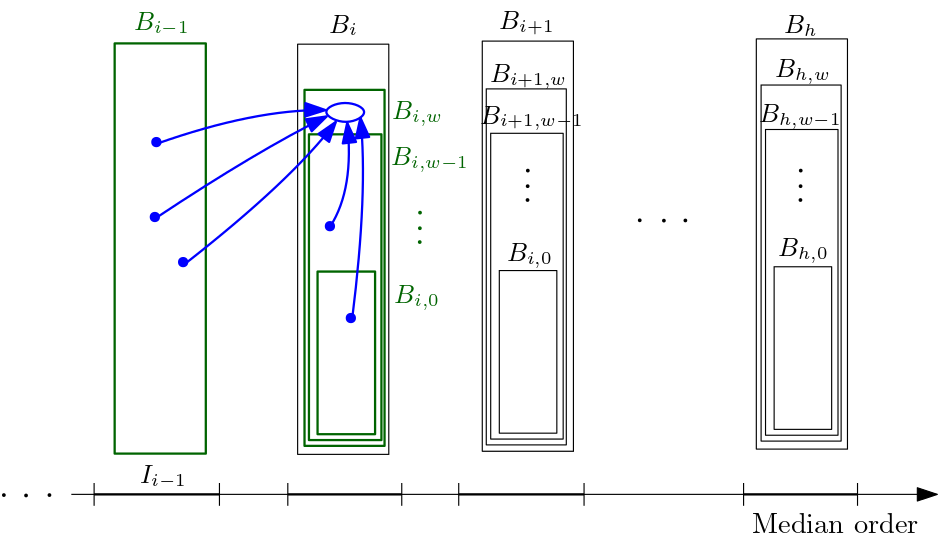}
        \caption{We find this structure inductively, starting with $B_h$. After having found $B_i, \dots B_h$ and the subsets $B_{i', 0} \subseteq B_{i', 1} \subseteq \dots \subseteq B_{i', w} \subseteq B_{i'}$ for $i < i' \leq h$, \cref{lemma:backward_step} will allow us to simultaneously find $B_{i, 0}  \subseteq B_{i, 1}  \subseteq \dots \subseteq B_{i, w} \subseteq B_i$ and $B_{i-1}$.}
        \label{subfig:lemma43}
    \end{subfigure}
    \caption{We find the embedding structure in the host tournament $T$ (\cref{subfig:embedding_structure}) layer-by-layer using \cref{lemma:backward_step} (\cref{subfig:lemma43}).}
    \label{fig:embedding_structure}
\end{figure}
Let $G$ be a digraph with graded bandwidth $w$ and maximum degree $\Delta$, and let $V_1, \dots, V_H$ be the layers of a graded partition of $G$.
To prove an upper bound on $\vv{r}(G)$, we fix a tournament $T$ of sufficient size and wish to find an embedding of $G$ into $T$.
We achieve this in two steps.
First, we want to find a suitable structure in the host tournament that will make this embedding simpler.
More specifically, we would like to find sets $A_1, \dots, A_H \subseteq V(T)$ such that for any $i \in [H]$ almost all $\Delta$-tuples of vertices in $A_{i-w} \cup \dots \cup A_{i-1}$ have a large common out-neighhorhood in $A_i$.
We achieve that using the dependent random choice technique combined with properties of the median order of $T$, i.e., an ordering of the vertices of $T$ that maximizes the number of forward edges. 
In the second step, we then want to embed the layers $V_1, \dots, V_H$ of $G$ one by one, such that each $V_i$ lands in the corresponding $A_i$.
Since the edge structure between the $A_i$'s in $T$ and between the $V_i$'s in $G$ resemble each other, this will in fact be possible using a careful application of the Lov\'asz local lemma, similar to the argument used in \cite{MR3548291}.

Finding the sets $A_1, \dots, A_H$ turns out to be quite challenging.
In fact, it turns out to be more convenient to find a slightly stronger structure.
We let $h = \frac{H}{w}$ and find disjoint sets $B_1, \dots, B_h$ together with subsets $B_{i, 0} \subseteq \dots \subseteq B_{i, w} \subseteq B_i$ for each $i \in [h]$ such that almost every $\Delta$-tuple in $B_{i-1} \cup B_{i, j-1}$ has a large common out-neighborhood into $B_{i, j}$ for every $i\in [h]$ and $j \in [w]$ (see \cref{fig:embedding_structure}).
Note that by taking $A_{iw + j} = B_{i+1, j}$ for each $0 \leq i \leq\frac{H}{w}$ and $1 \leq j \leq w$ we can recover our original plan with most pairs of the sets being disjoint.

We find the sets $B_i$ and $B_{i, j}$ inductively using the median order of $T$, starting with $i = h$.
At a given step, for some $i$, we have found the sets $B_i, \dots, B_h$ as well as the subsets $B_{i', j}$ for all $i' > i$ and $0 \leq j \leq w$.
Moreover, we have ensured that $B_i \subseteq I_i$ for some small interval $I_i$ of the median order of $T$.
Using the properties of the median order, we can find an interval $I_{i-1}$ not far away from $I_i$ and of roughly the same size, such that at least $1/3$ of the edges from $I_{i-1}$ to $B_i$ are directed from $I_{i-1}$ to $B_i$. 
Then, we simultanously find the sets $B_{i, 0} \subseteq \dots \subseteq B_{i, w} \subseteq B_i$ and the set $B_{i-1}$, making sure that $B_{i-1} \subseteq I_{i-1}$.

To achieve that, we apply the dependent random choice technique (for more information on this method, see the survey \cite{MR2768884}).
Put briefly, given two sets $L$ and $R$, with some constant fraction of edges going from $L$ to $R$, dependent random choice allows us to find a set $K \subseteq R$ of constant size $k$ with the property that $L' = N^-(K) \cap L$ is large and almost all $\Delta$-subsets of $L'$ have many common out-neighbors in $R$.
Given a $B_{i, j}$, this would allow us to find a $B_{i, j-1} \subseteq B_{i, j}$, where almost all $\Delta$-subsets of $B_{i, j-1}$ have large common out-neighborhood in $B_{i, j}$.
This is however not enough, because recall that we want to have that most of $\Delta$-subsets of $B_{i-1} \cup B_{i, j-1}$ have this property.

Therefore, we slightly modify the dependent random choice argument.
We notice that given some sets $A, L$ and $R$, again with a constant fraction of vertices going from $L$ to $R$, we can find a $K \subseteq R$ of constant size $k$ such that for $A' = N^-(K) \cap A$ and $L' = N^{-}(K) \cap L$ we again have that $L'$ is large and almost all $\Delta$-subsets of $A' \cup L'$ have a large common out-neighborhood in $R$.
By setting $A = I_{i-1}$ we are therefore almost done; what remains to do is to control the size of $A'$ in some way.
Luckily, this we can do by yet another application of dependent random choice.
\begin{figure}
    \centering
    \begin{subfigure}[h]{0.45\linewidth}
        \includegraphics[width=\linewidth]{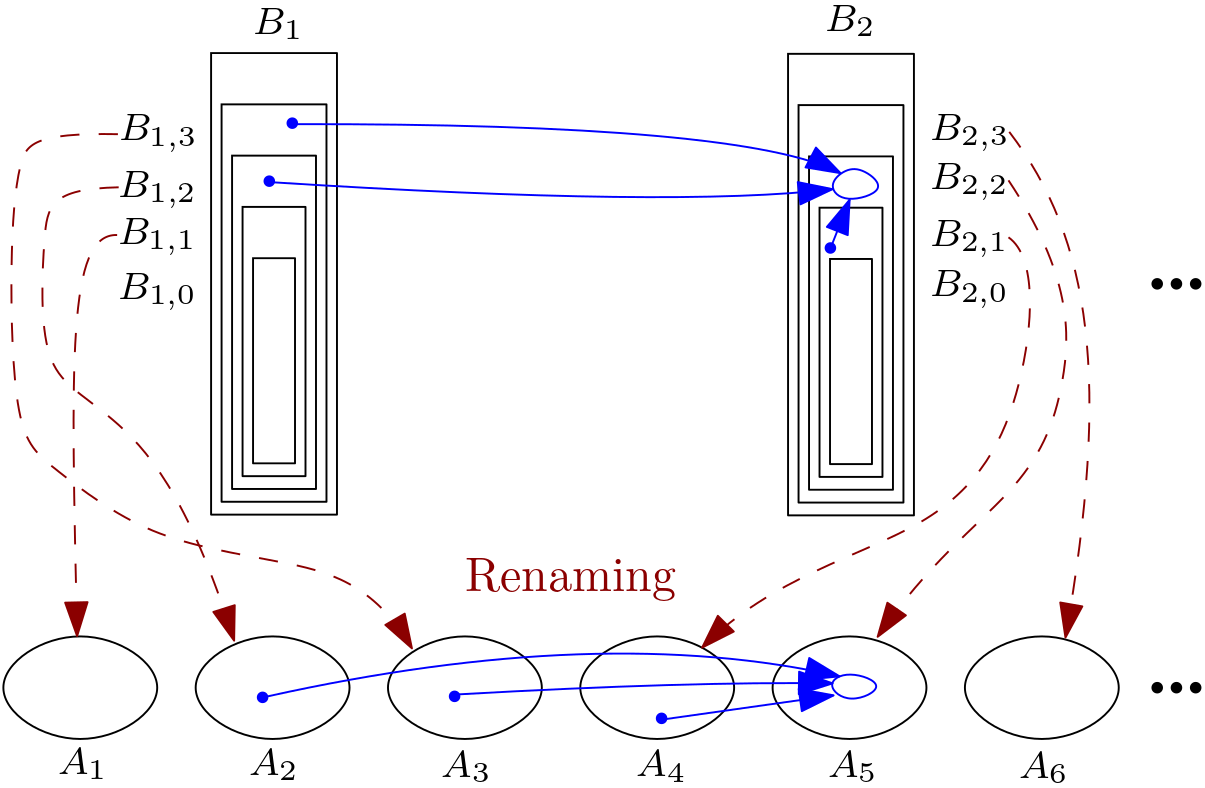}
        \caption{Renaming of the sets $B_{i, j}$ from \cref{theorem:embedding_structure} into (non-disjoint) $A_1, \dots, A_H$ such that almost all $\Delta$ subsets of $A_{i - w} \cup \dots \cup A_{i-1}$ have large common out-neighborhood in $A_i$ for each $i \in [H]$.}
        \label{subfig:embedding_renaming}
    \end{subfigure}
    \hfill
    \begin{subfigure}[h]{0.45\linewidth}
        \includegraphics[width=\linewidth]{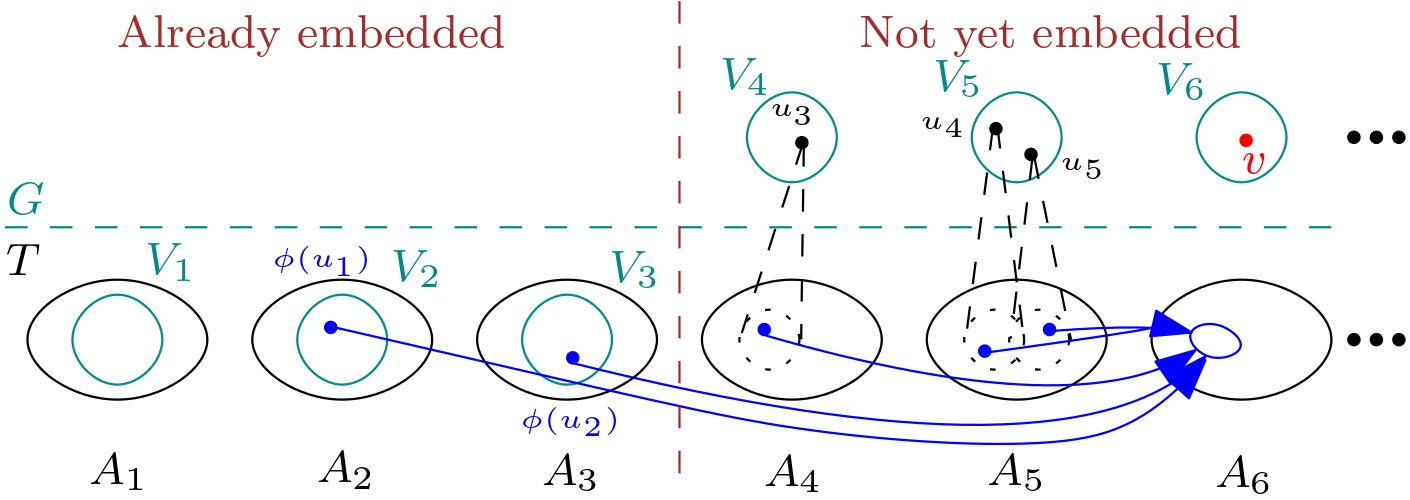}
        \caption{We embed the sets $V_i$ into the sets $A_i$ one by one. At each step we make sure that for all non-embedded sets $V_j$ and all $v \in V_j$ the following holds.
        Let $u_1, \dots, u_{\Delta}$ be the in-neighbors of $v$ such that we have already embedded $u_1, \dots, u_l$ for some $l \leq \Delta$.
        Then, we want that for almost all possible choices for the embedding of the remaining $u_{l+1}, \dots, u_{\Delta}$, we have that $\phi(u_1), \dots, \phi(u_{\Delta})$ have a large common out-neiborhood in $V_j$, i.e., we will have a lot of space to put $v$ into.}
        \label{subfig:embedding_step}
    \end{subfigure}
    \caption{To prove \cref{theorem:main_theorem}, we start by renaming the sets from \cref{theorem:embedding_structure} into $A_1, \dots, A_H$ (\cref{subfig:embedding_renaming}) and then embed the sets $V_1, \dots, V_H$ into the sets $A_1, \dots, A_H$ one by one (\cref{subfig:embedding_step}).}
    \label{fig:main_theorem_proof}
\end{figure}

More specifically, in each iteration, having found $B_i$ and $I_{i-1}$, we first apply dependent random choice to find $B_{i, w}$ such that almost all subsets $S \subseteq B_{i, w}$ of size $w \cdot k$ have many common in-neighbors in $I_{i-1}$.
Then, we iteratively find subsets $K_j \subseteq B_{i, j}$ of size $k$ each, as described above, and set $B_{i, j-1} = N^-(K_j) \cap B_{i, j}$.
Most importantly, every time we make sure that $K_j \cup \dots \cup K_w$ has many common in-neighbors in $I_{i-1}$ --- and that this still holds for most $S \cup K_j \cup \dots \cup K_w$, where $S \subseteq B_{i, j-1}$.
At the end, setting $B_{i-1} = N^-(K_1 \cup \dots \cup K_w) \cap I_{i-1}$ finishes the step.

By repeating the same argument, we can therefore find the sets $B_i$ and $B_{i, j}$ as described --- and what remains to show is that we can in fact embed our $G$ into them.
For that, it is more convenient to get back to our original plan with sets $A_1, \dots, A_H$ --- which as we have seen can be recovered from the $B_{i, j}$'s.
We want to embed each of the layers $V_1, \dots, V_H$ of $G$ into the respective $A_i$'s --- starting with embedding $V_1$ into $A_1$ and embedding respectively the next $V_i$ in each step.
We have to do it carefully --- when embedding $V_i$ into $A_i$ we not only want to respect the embedding of $V_1, \dots, V_{i-1}$ that we have already found but also guarantee that we will be able to continue the process for $V_{i+1}$, $V_{i+2}$ etc.

So suppose we have found an embedding $\phi$ of the layers $V_1,\dots,V_{i}$ into $A_1,\dots,A_{i}$ and let $v \in V_j$ for $j > i$ be some vertex that we have not embedded yet.
Moreover, let $u_1, \dots, u_\Delta$ be the in-neighbors of $v$ in $G$.
If $u_1, \dots, u_{\Delta} \in V_1 \cup \dots \cup V_i$, that is, we have already embedded all the in-neighbors of $v$, we would like that there are a lot of potential vertices in $V_j$ where we can put $v$.
In other words, we want to make sure that the common out-neighborhood of $\phi(u_1), \dots, \phi(u_\Delta)$ in $A_j$ is large.
Otherwise, if we only already embedded some (possibly empty) part of the in-neighbors of $v$, say $u_1, \dots, u_\ell$ for some $\ell < \Delta$, then we want to make sure that most possible choices of $\phi(u_{\ell+1}), \dots, \phi(u_\Delta)$ give a lot of space to put $v$.
Using the Lov\'asz local lemma in a similar manner to \cite{MR3548291},
we will prove that if the above conditions are satisfied, then we will be able to embed $V_{i+1}$, while making sure that these conditions are satisfied again.
Continuing inductively, we are thus able to embed the whole $G$ into $T$.

\vspace{5px}
The remainder of this section is organized as follows.
First in \cref{subsection:preliminaries} we define the concepts used in the proof of \cref{theorem:main_theorem}.
Then, the proof is split into two parts.
In \cref{subsection:structure_theorem}, we first find the sets $B_i$ and $B_{i, j}$ in the host tournament.
Then, in \cref{subsection:embedding_into_structure} we show that we can indeed embed our digraph with bounded graded bandwidth into this structure, thus proving \cref{theorem:main_theorem}.

\subsection{Preliminaries}\label{subsection:preliminaries}
In this section we define various concepts used in the proof of \cref{theorem:main_theorem} --- each in the corresponding subsection --- and prove a couple of simple observations related to them.

\subsubsection{Upward closed sets and \texorpdfstring{$k$}{k}-density}
Throughout the paper, we will often have that almost all subsets of vertices of a given size $k$ have a large common neighborhood in some other set of vertices. In order to more conveniently work with such conditions, we introduce the following simple notions.

\begin{definition}[Upward closed, $k$-density]
    Let $A$ be a set, $1 \leq k \leq |A|$, $0\leq \delta \leq 1$ and ${\cal F} \subseteq 2^A$.
    We say that ${\cal F}$ is \emph{upward closed} if for each $Y \in {\cal F}$ and each $Y \subseteq X \subseteq A$ we have $X \in {\cal F}$.
    We say that ${\cal F}$ has \emph{$k$-density} $\delta$ if $|{\cal F} \cap \binom{A}{k}| = \delta \binom{|A|}{k}$.
\end{definition}

The main property of upward closed sets is that knowing that their $k$-density is at most $\delta$ immediately tells us that for each $1 \leq i \leq k$ their $i$-density is at most $\delta$ too.
This will be useful for us, since it immediately implies, that the probability that a random set of $k$ elements of $A$ drawn uniformly at random with replacement has probability at most $\delta$ to be ``bad".

\begin{observation}[Upward closed property]\label{observation:up_closed_density}
    Let $A$ be a set, $1 \leq k \leq |A|$ and ${\cal F} \subseteq 2^A$ be upward closed with $k$-density at most $\delta$.
    Then ${\cal F}$ has $i$-density at most $\delta$ for each $1 \leq i \leq k$.
\end{observation}
\begin{proof}
    For $Y \subseteq A$ with $|Y| = i$, $Y \in {\cal F}$ implies $X \in {\cal F}$ for all $\binom{|A|-i}{k-i}$ $k$-subsets of $X$ containing $Y$.
    Moreover, the number of pairs consisting of an $X \in {\cal F}$ of size $k$ and a $Y \subseteq X$ of size $i$ is at most $\delta \binom{|A|}{k} \binom{k}{i}$.
    Therefore, the number of $Y \in {\cal{F}}$ of size $i$ is at most
    \[
        \frac{\delta \binom{|A|}{k} \binom{k}{i}}{\binom{|A| - i}{k-i}} = \delta \binom{|A|}{i}.\qedhere
    \]
\end{proof}

\subsubsection{Communities}
In the proof of \cref{theorem:main_theorem}, we will make use of the dependent random choice technique --- which given two sets $L$ and $R$ with positive edge-density between them, will allow us to find a subset $A \subseteq L$ such that all but a few subsets $S \subseteq A$ of size $\Delta$ have more than $s$ common neighbors in $R$.
Moreover, at later stages, we will want to have some elements $C$  of $A$ fixed --- for example because we embedded some vertices of $G$ into them --- while still wanting that for almost all $S \subseteq A$ of size $\Delta'$ the set $C \cup S$ has a large common neighborhood in $R$.
To increase readability and avoid writing out the whole ``all but at most $e$ subsets..." every time, we define the notion of a \emph{community}.
\begin{definition}
    Let $G$ be a graph, $A, C, R \subseteq V(G)$.
    We say that $(A, C)$ is an \emph{$(R, \Delta, s, e)$-community} if $|N(C) \cap R| \geq s+1$, $|A\setminus C| \geq \Delta$ and for all but at most $e$ subsets $S \subseteq A \setminus C$ of size $|S| = \Delta$ we have $|N(C \cup S) \cap R| \geq s+1$.
\end{definition}
\noindent Note that for the first example above, we would write $C = \emptyset$, i.e., that $(A, \emptyset)$ is an $(R, \Delta, s, e)$-community.
Notice also that in case $\Delta = 0$, this definition simply means that $|N(C) \cap R| \geq s+1$.

For digraphs, we define the notion of community analogously, this time distinguishing between the in- and out-neighborhood.
\begin{definition}
    Let $T$ be a digraph and $A, C, R \subseteq V(T)$.
    We say that $(A, C)$ is  an \emph{$(R, \Delta, s, e)$-in-community} (respectively \emph{$(R, \Delta, s, e)$-out-community}) if $|N^-(C) \cap R| \geq s+1$ (respectively $|N^+(C) \cap R| \geq s+1$), $|A\setminus C| \geq \Delta$ and for all but at most $e$ subsets $S \subseteq A \setminus C$ of size $|S| = \Delta$ we have $|N^-(C \cup S) \cap R| \geq s+1$ (respectively $|N^+(C \cup S) \cap R| \geq s+1$).
\end{definition}

By definition, we can take any subset $A' \subseteq A$ and remain a community with the same parameters.
\begin{observation}\label{observation:community_subset}
    If $(A, C)$ is an $(R, \Delta, s, e)$-community then $(A', C)$ is a $(R, \Delta, s, e)$-community for any $A' \subseteq A$ such that $|A' \setminus C| \ge \Delta$.
\end{observation}

As described above, we will often want to fix some subset $C \subseteq A$ and still get that almost all $C \cup S$ have a large common neighborhood in $R$.
For example, as we will be embedding the vertices of our digraph $G$ in stages, for some vertex $v \in V(G)$ the set $C$ will be the set of in-neighbors of $v$ that we have already embedded.
Importantly, as we keep embedding the vertices of $G$, this set $C$ will grow, and at each stage we will want that there are many possible ways to pick a $S' \subseteq A$ such that $(A, C \cup S')$ is still a good community.
The following observation allows us to achieve exactly that.
\begin{observation}\label{observation:community_extension_density}
    Let $G$ be a graph and $A, C, R \subseteq V(G)$ be subsets of its vertices.
    Moreover, let $\Delta_1, \Delta_2, m, s \in \mathbb{N}$ such that $\Delta_1 + \Delta_2 \leq m \leq |A \setminus C|$ and $\delta_1, \delta_2 \ge 0.$
    Suppose that $(A, C)$ is an $(R, \Delta_1 + \Delta_2, s, \delta_1 \delta_2 \binom{m}{\Delta_1 + \Delta_2})$-community.
    Let ${\cal F}$ be the set of subsets $S' \subseteq A \setminus C$ such that $(A, C \cup S')$ is not an $(R, \Delta_2, s, \delta_2 \binom{m-\Delta_1}{\Delta_2})$-community.
    Then ${\cal F}$ is upward closed and $|{\cal F} \cap \binom{A\setminus C}{\Delta_1}| \leq \delta_1 \binom{m}{\Delta_1}$.
    In particular, ${\cal F}$ has $\Delta_1$-density at most $\delta_1$.
\end{observation}
\begin{proof}
    Since for any $X, Y \subseteq V(G)$ we have that $N(X \cup Y) \subseteq N(X)$ we immediately get that ${\cal F}$ is upward closed.
    Suppose now for contradiction that $|{\cal{F}}\cap \binom{A\setminus C}{\Delta_1}| > \delta_1 \binom{m}{\Delta_1}$.
    Then the number of subsets $S \subseteq A \setminus C$ of size $\Delta_1 + \Delta_2$ with $|N(C \cup S) \cap R| \leq s$ is larger than
    \[
         \delta_1 \delta_2 \frac{\binom{m}{\Delta_1}\binom{m-\Delta_1}{\Delta_2}}{\binom{\Delta_1 + \Delta_2}{\Delta_1}} = \delta_1 \delta_2 \binom{m}{\Delta_1 + \Delta_2},
    \]
    which is a contradiction to $(A, C)$ being an $(R, \Delta_1 + \Delta_2, s, \delta_1 \delta_2 \binom{m}{\Delta_1 + \Delta_2})$-community.
    The $\Delta_1$-density of ${\cal F}$ then follows from $m \leq |A \setminus C|$.
\end{proof}
\begin{remark}
    Observations \ref{observation:community_subset} and \ref{observation:community_extension_density} hold for in- and out-communities as well.
\end{remark}
\subsubsection{Median Order}
Finally, an important ingredient for finding the embedding structure in $T$ will be the median order \cite{MR1791347} of the host tournament $T$.
\begin{definition}
    Let $T$ be a tournament on $N$ vertices.
    An ordering $v_1, \dots, v_N$ of the vertices of $T$ is called a \emph{median order} if it maximizes the number of forward edges, i.e., pairs $i < j$ with $v_iv_j \in E(T)$ among all possible orderings of the vertices of $T$.
    Given a median order $v_1, \dots, v_N$ and $1 \leq i < j \leq N+1$ we write $[i, j)$ for $\{v_i, v_{i+1}, \dots, v_{j-1}\}$.
\end{definition}
\noindent The median order is an elementary, yet powerful tool in the study of tournaments and has for example been used in the recent work around Sumner's conjecture; we refer the reader to \cite{MR1791347} for more information.

We will use the median order to, starting with some set $B$ contained in a given interval $I$ of the median order, find another interval $I'$ of size roughly the same as $I$ such that many edges go from $I'$ to $B$.
This is implied by the following standard observation about the median order.
\begin{observation}\label{observation:median_order}
    Let $T$ be a tournament on $N$ vertices with a median order $v_1, \dots, v_N$.
    Let $j, k, a \in \mathbb{N}$ satisfy $j - ka \geq 1$ and $j + a + 1 \leq N$.
    Then, for $I = [j - ka, j)$ and any $A \subseteq [j, j+a)$ we have $d(I, A) \geq \frac{k-1}{2k}$.
\end{observation}
\begin{proof}
    Notice that since $v_1, \dots, v_N$ is a median order, at least half of the vertices from $[j - ka, i)$ are in-neighbors of $v_i$ for each $j \leq i < j+a$.
    Therefore, fore every $v_i \in [j, j+a)$, we have
    \[
        d^-(v_i, I) = d^-(v_i, [j - ka, i)) - d^-(v_i, [j, i)) \geq \frac{ka + i - j}{2} - (i - j) = \frac{ka - (i-j)}{2} \geq \frac{(k-1)a}{2}.
    \]
    In particular, since $A \subseteq [j, j+a)$, we get
    \[
        d(I, A) = \frac{\sum_{v \in A} d^-(v, I)}{ka \cdot |A|} \geq  \frac{|A| (k-1) a}{2ka \cdot |A|} = \frac{k-1}{2k}.\qedhere
    \]
\end{proof}

\subsection{Finding the embedding structure in the host tournament}\label{subsection:structure_theorem}

In this section, we will find the embedding structure in the host tournament $T$.
We first state \cref{theorem:embedding_structure}, which formally defines this structure and says that we can indeed find it.
In our application, the parameters $\Delta$ and $w$ will be the maximum degree and the graded bandwidth of our digraph $G$, $h$ will be chosen such that $h \cdot w \approx H$, while $s_1, \dots, s_h$ will roughly correspond to the sizes of the layers $V_1, \dots, V_H$ of $G$: $s_1$ will correspond to the maximum among $|V_1|, \dots, |V_w|$, $s_2$ will correspond to the maximum among $|V_{w+1}|, \dots, |V_{2w}|$, and so on.

\begin{theorem}\label{theorem:embedding_structure}
    Let $\Delta, w, h, s_1, \dots, s_h \in \mathbb{N}$, such that $s_i \geq \max\{\frac{s_{i-1}}{2}, \frac{s_{i+1}}{2}\}$ for each $i$, and let $0< \delta <1$.
    For each $i \in [h]$ define:
    \begin{itemize}
        \item $m_i = \delta^{-\frac{1}{\Delta}} 3^{30 \Delta w} s_i$,
        \item $b_i = 3^{6\Delta w} m_i$, \emph{and}
        \item $a_i = 3^{10\Delta w} b_i$.
    \end{itemize}
    Let $T$ be a tournament on $N \geq \sum_{i=1}^h 6a_i$ vertices.
    Then there exist disjoint $B_1, \dots, B_h \subseteq V(T)$ and $B_{i,0} \subseteq \dots \subseteq B_{i, w} \subseteq B_i$ for each $i \in [h]$ such that the following holds. 
    For each $i \in [h]$ and $j \in [w]$ we have that
    \begin{enumerate}
        \item $m_i \leq |B_{i,j}| \leq b_i$, \emph{and}
        \item $(B_{i-1} \cup B_{i, j-1}, \emptyset)$ is a $(B_{i,j}, \Delta, s_i, \delta (\frac{s_i}{2b_i})^{\Delta} \binom{m_i/3}{\Delta})$-out-community.
    \end{enumerate}
\end{theorem}
We remark that in \cref{theorem:embedding_structure}, we have $\frac{s_i}{b_i}=\delta^{1/\Delta}3^{-36\Delta w}$ for all $i$. Thus, the final conclusion of the theorem could be rewritten to not involve the quantity $\frac{s_i}{2b_i}$, as this quantity is the same for all $i$. 

Before giving a proof of \cref{theorem:embedding_structure} in \cref{subsection:proof_of_structure_theorem}, in \cref{subsection:backward_step} we first state and prove \cref{lemma:backward_step}, which will allow us to find one layer of the structure at a time.

\subsubsection{Extending the structure by one layer}\label{subsection:backward_step}
The main tool while finding the next layer of our structure will be a dependent random choice argument, which is the following \cref{lemma:dependent_random_choice}.
As described before, given some sets $A$, $L$ and $R$, with some positive edge-density between $L$ and $R$, the lemma will allow us to find a small subset $K \subseteq R$ such that $N(K) \cap L$ is large and for $A' = N(K) \cap A$ and $L' = N(K) \cap L$, almost all $\Delta$-subsets of $L' \cup A'$ have a large common neighborhood in $R$.
Moreover, to have control of the size of $A'$, we will guarantee that $K \notin {\cal{F}}$, for some not too $k$-dense set ${\cal{F}}$, which in our case will simply be the set of $S \subseteq R$ such that $N(S) \cap A$ is small.
\begin{lemma}\label{lemma:dependent_random_choice}
    Let $G$ be a graph and let $A, L, R \subseteq V(G)$.
    Moreover, let $k, s, \Delta \in \mathbb{N}$ and $0 \leq d \leq 1$ be such that $ \frac{\sum_{v \in L} d(v, R)}{|L|\cdot|R|} \geq d$.
    Let $\mathcal{F} \subseteq 2^R$ be an up-closed family with $k$-density at most $\delta$ for some $\delta < d^k/2$.
    Then, there exists $K \subseteq R$ of size $|K| \leq k$ such that
    \begin{enumerate}
        \item $K \notin \cal{F}$,
        \item $|N(K) \cap L| \geq \frac{d^k - \delta}{2}|L|$, \emph{and}
        \item $(N(K) \cap (L \cup A), \emptyset)$ is a $(R, \Delta, s, \frac{4}{d^k - \delta} \binom{|L \cup A|}{\Delta} (\frac{s}{|R|})^k)$-community.
    \end{enumerate}
\end{lemma}
\begin{proof}
    Pick a set $K$ of $k$ vertices from $R$ uniformly at random with repetitions and let $M = N(K) \cap L= \{ u \in L : K \subseteq N(u) \cap R \}$.
    Let $X = |M|$ and let $E$ be the event that $K \notin \cal{F}$.
    Note that by \cref{observation:up_closed_density} we get that $\Pr[E] \geq 1 - \delta$.
    By Jensen's inequality we get
    \[
        \mathbb{E}[X] = \sum_{v \in L} \left(\frac{d(v,R)}{|R|}\right)^k \geq |L|d^k
    \]
    and since $X \leq |L|$ and $\mathbb{E}[X | E] \geq \mathbb{E}[X] - \Pr[\bar{E}] \mathbb{E}[X | \bar{E}]$ we get that
    \[
        \mathbb{E}[X | E] \geq |L|(d^k - \delta).
    \]
    Let now $Y$ be the random variable counting the number of subsets $S \subseteq N(K) \cap (L \cup A)$ of size $|S| = \Delta$ with at most $s$ common neighbors in $R$.
    For a given such $S \subseteq L \cup A$, the probability that it is a subset of $N(K)$ is at most $(\frac{|N(S) \cap R|}{|R|})^k \leq (\frac{s}{|R|})^k$ and therefore
    \[
        \mathbb{E}[Y] \leq \binom{|L \cup A|}{\Delta} \left( \frac{s}{|R|}\right)^k.
    \]
    Moreover, since $Y \geq 0$, we get that $\mathbb{E}[Y | E] \leq (1-\delta)^{-1}\mathbb{E}[Y] \leq 2 \mathbb{E}[Y]$.

    Let us suppose for a moment that $\mathbb{E}[Y | E] > 0$.
    By the linearity of expectation we know that
    \[
    \mathbb{E}\left[X - \frac{\mathbb{E}[X | E]}{2\mathbb{E}[Y | E]} Y - \frac{\mathbb{E}[X | E]}{2} ~\middle|~ E\right] = 0.
    \]
    Therefore, there exists a choice of $K \notin {\cal{F}}$ with the corresponding $M$ for which

    \[ X - \frac{\mathbb{E}[X | E]}{2\mathbb{E}[Y | E]} Y - \frac{\mathbb{E}[X | E]}{2} \ge 0. \]
    Fix such a choice.
    Then,
    \[
        |M| = X \geq \frac{\mathbb{E}[X | E]}{2} \geq |L| \frac{d^k - \delta}{2},
    \]
    and since $|X| \leq |L|$, it also holds that
    \[
        Y \leq \frac{2X}{\mathbb{E}[X | E]} \mathbb{E}[Y | E] \leq \frac{4}{d^k - \delta} \binom{|L \cup A|}{\Delta} \left(\frac{s}{|R|}\right)^k.
    \]
    In particular, this choice of $K$ concludes the proof.

    It remains to consider the case when $\mathbb{E}[Y | E] = 0$.
    In this case, since $Y$ is always non-negative, it must take the value $0$ with probability $1$ conditioned on $E$.
    Thus, if we select any $K \notin \cal{F}$ for which the corresponding $X$ is at least $\mathbb{E}[X | E]$, we can again conclude the proof.
\end{proof}

We are now ready to state and prove \cref{lemma:backward_step}, which will allow us to find one layer of our structure at a time.
More specifically, having found $B_i$ in the current step, as well as an interval $I_{i-1}$ of the median order with many edges going from $I_{i-1}$ to $B_i$, \cref{lemma:backward_step} will allow us to find $B_{i, 0} \subseteq B_{i, 1} \subseteq \dots \subseteq B_{i, w} \subseteq B_i$ and $B_{i-1} \subseteq I_{i-1}$ such that for each $j \in [w]$ almost all $\Delta$-subsets of $B_{i-1} \cup B_{i, j-1}$ have many common out-neighbors in $B_{i, j}$.
\begin{lemma}\label{lemma:backward_step}
    Let $\Delta, w, s \in \mathbb{N}$ and let $T$ be a tournament.
    Let $I, B \subseteq V(T)$ be disjoint such that $d(I, B) \geq 1/3$ and $|B| \geq 3 \cdot 3^{6\Delta w}$.
    Then there exist $A \subseteq I$ and $B_0 \subseteq B_1 \subseteq \dots \subseteq B_w \subseteq B$ such that
    \begin{enumerate}
        \item $|A| = 3^{-10\Delta w} |I|$,
        \item for $j \in [w]$ we have that $3^{-3\Delta w} |B| \geq |B_j| \geq 3^{-6\Delta w}|B|$, \emph{and}
        \item for each $j \in [w]$ we have that $(A \cup B_{j-1}, \emptyset)$ is a $(B_j, \Delta, s, (3^{8\Delta w} \frac{s}{|B|})^{2\Delta} \binom{|I \cup B_{j-1}|}{\Delta})$-out-community.
    \end{enumerate}
\end{lemma}
\begin{proof}
    Let $b = 3^{-10 \Delta w} |I|$, $\delta = 3^{-3\Delta}$ and $m = 3^{-6\Delta w}|B|$.
    We first apply \cref{lemma:dependent_random_choice} with $L_{\ref{lemma:dependent_random_choice}} = B$, $R_{\ref{lemma:dependent_random_choice}} = I$, $A_{\ref{lemma:dependent_random_choice}} = \emptyset$, $s_{\ref{lemma:dependent_random_choice}} = b$ and $k_{\ref{lemma:dependent_random_choice}} = \Delta_{\ref{lemma:dependent_random_choice}} = 2\Delta w$ 
    to find a $K \subseteq I$ such that for $B_w = N^+(K) \cap B$ we have that $|B_w| \geq \frac{3^{-2\Delta w}}{2} |B| \geq 3^{-3\Delta w} |B|$ and that $(B_w, \emptyset)$ is an $(I, 2\Delta w, b, \delta^w \binom{m}{2\Delta w})$-in-community.
    For the last parameter, the bound on the number of bad $(2 \Delta w)$-subsets of $B_w$, we used that
    \begin{align*}
        4 \cdot 3^{2\Delta w} \binom{|B|}{2 \Delta w} \left(\frac{b}{|I|}\right)^{2 \Delta w} &\leq 4 \cdot 3^{2\Delta w} |B|^{2 \Delta w} (3^{-10\Delta w})^{2 \Delta w} \\
        & \leq 3^{-3\Delta w} \left(\frac{3^{-6\Delta w} |B|}{2\Delta w}\right)^{2\Delta w} \\
        & \leq \delta^w \binom{m}{2 \Delta w}.
    \end{align*}
    Moreover, if $|B_w| > 3^{-3\Delta w}|B|$, we can simply restrict it to some arbitrary subset of size exactly $3^{-3\Delta w}|B|$, and maintain the community structure.

    We will now find sets $\emptyset = K_w \subseteq K_{w-1} \subseteq \dots \subseteq K_0 \subseteq B_w$ such that for $B_i = N^-(K_i) \cap B_w$,
    and $A_i = N^-(K_i) \cap I$
    the following conditions hold.
    For each $i = 0, \dots, w$ and $j = 1, \dots, w$ we want that
    \begin{enumerate}
        \item $|K_i| \leq 2\Delta(w-i)$,
        \item $|B_{j-1}| \geq 3^{-3\Delta}|B_{j}| \geq m$,
        \item\label{condition:continuation_possible} $(B_j, K_j)$ is an $(I, 2\Delta j, b, \delta^j \binom{m}{2 \Delta j})$-in-community, \emph{and}
        \item $(A_{j-1} \cup B_{j-1}, \emptyset)$ is a $(B_j, \Delta, s, (3^{8 \Delta w} \frac{s}{|B|})^{2 \Delta} \binom{|I \cup B_{j-1}|}{\Delta})$-out-community.
    \end{enumerate}

    Clearly, our initial choice of $B_w$ satisfies the above conditions upon setting $K_w=\emptyset$.
    Let us now suppose that for some $i \in [w]$ we have found appropriate $K_{i'}$ for all $i' \geq i$.
    We want to show how to find a suitable $K_{i-1}$.

    To that end, let $\mathcal{F} \subseteq 2^{B_i}$ be the collection of all subsets $S \subseteq B_i$ such that $(B_i, K_i \cup S)$ is not an $(I, 2\Delta (i-1), b, \delta^{i-1} \binom{m}{2\Delta (i-1)})$-in-community.
    By Condition \ref{condition:continuation_possible} and \cref{observation:community_extension_density} we get that $\cal{F}$ is upward closed with $2\Delta$-density at most $\delta = 3^{-3\Delta}$.

    Now, since $|B_i| \geq m \geq 3$, we get that $\frac{\sum_{v \in B_i} d(v, B_i)}{|B_i|^2} \geq 1/3$.
    Therefore, by \cref{lemma:dependent_random_choice} with $L_{\ref{lemma:dependent_random_choice}} = R_{\ref{lemma:dependent_random_choice}} = B_i$, $A = A_i$ and $k_{\ref{lemma:dependent_random_choice}} = 2 \Delta$ we can find a set $K' \subseteq B_i, K' \not\in \mathcal{F}$ such that for $K_{i-1} = K_i \cup K'$, and $A_{i-1}, B_{i-1}$ defined as above, we get
    \begin{itemize}
        \item $|K_{i-1}| \leq |K_i| + 2\Delta \leq 2\Delta (w - i +1)$,
        \item $|B_{i-1}| \geq \frac{3^{-2\Delta} - 3^{-3\Delta}}{2} |B_i| \geq 3^{-3\Delta}|B_i|$, \emph{and}
        \item $(A_{i-1} \cup B_{i-1}, \emptyset)$ is a $(B_i, \Delta, s, \frac{4}{3^{-2\Delta} - \delta} \binom{|B_i \cup A_i|}{\Delta} (\frac{s}{|B_i|})^k)$-out-community.
    \end{itemize}
    Since $K' \notin \cal{F}$, condition $\ref{condition:continuation_possible}$ is satisfied.
    Moreover, since $|B_i| \geq m$ and $A_i \subseteq I$, we get that
    \begin{align*}
        \frac{4}{3^{-2\Delta} - \delta} \binom{|B_i \cup A_i|}{\Delta} \left(\frac{s}{|B_i|}\right)^k &\leq 4 \cdot 3^{3\Delta} \left(3^{6\Delta w} \frac{s}{|B|}\right)^{2\Delta} \binom{|I \cup B_i|}{\Delta} \\
        & \leq \left(3^{8\Delta w} \frac{s}{|B|}\right)^{2\Delta} \binom{|I \cup B_i|}{\Delta}.
    \end{align*}
    Thus, all the remaining conditions are satisfied as well.

    We can therefore continue this argument all the way up through $i = 0$.
    At that point, we notice that by condition \ref{condition:continuation_possible} we have $|A_0| \geq b$.
    Since $A_0 \subseteq A_i$ for each $i = 0, \dots, w$, by taking an arbitrary $A \subseteq A_0$ of size $b$ we are able to find the subsets $A \subseteq I$ and $B_0 \subseteq \dots \subseteq B_w \subseteq B$ satisfying the conditions of the lemma.
\end{proof}

\subsubsection{Proof of Theorem \ref{theorem:embedding_structure}}\label{subsection:proof_of_structure_theorem}

Now, proving \cref{theorem:embedding_structure} is simply an inductive application of \cref{lemma:backward_step}, while being careful about the parameters.
At each step, we find the interval $I$ for \cref{lemma:backward_step} using the median order of $T$.
\begin{proof}[Proof of \cref{theorem:embedding_structure}]
    Fix $\Delta, w, h, s_1, \dots, s_h \in \mathbb{N}$ and $0<\delta <1$ and let $m_i, b_i$ and $a_i$ be defined as in the statement of the theorem.
    It suffices to show that the statement holds for any tournament of size $N = \sum_{i=1}^h 6a_i$.
    Therefore, let $T$ be such tournament and let $v_1, \dots, v_N$ be its median order.

    We will find the $B_i$'s together with the corresponding $B_{i, j}$'s iteratively using the median order of $T$.
    We start by placing $B_w$ at the end of the median order.
    At each step, we will make sure that each $B_i$ is contained in the interval of the median order of size $a_i$ and will search for $B_{i-1}$ in an interval immediately preceding it.
    Throughout, we will write $o_i$ for the starting index of the interval of the median order of $T$ containing $B_i$, i.e., we will have $B_i \subseteq [o_i, o_i + a_i)$.
    At each step we will guarantee that $o_i - 6a_{i-1} \leq o_{i-1} \leq o_i - a_{i-1}$, and thus we will have that the $B_i$'s are disjoint and $o_1 \geq 1$.

    We start by setting $o_h = N - a_h + 1$ and $B_h = [o_h, N+1)$.
    Suppose now that for some $i=2, \dots, h$ we have defined the $o_{i'}$'s and the $B_{i'}$'s for all $i \leq i' \leq h$ and the $B_{i'', j}$'s for all $i < i'' \leq h$ and $0 \leq j \leq w$.
    Before finding $B_{i, 0} \subseteq \dots \subseteq B_{i, w}$ and $B_{i-1}$, we first want to find a suitable $o_{i-1}$.
    We let $I' = [o_i - 6a_{i-1}, o_i)$ and notice that by \cref{observation:median_order}, and using that $a_{i-1} \geq \frac{a_i}{2}$ and $B_i \subseteq [o_i, o_i + a_i)$, we get that $d(I', B_i) \geq 1/3$.
    Now, by averaging, for at least one $\ell \in [6]$ we have $d([o_i - \ell a_{i-1}, o_i - (\ell-1)a_{i-1})) \geq 1/3$.
    For such a choice of $\ell$, we set $o_{i-1} = o_i - \ell a_{i-1}$ and write $I_{i-1} = [o_{i-1}, o_{i-1} + a_{i-1})$.

    We now apply \cref{lemma:backward_step} with $I_{\ref{lemma:backward_step}} = I_{i-1}$, $B_{\ref{lemma:backward_step}} = B_{i}$ and $s_{\ref{lemma:backward_step}} = s_{i}$ which gives us $B_{i-1} \subseteq I_{i-1}$ and $B_{i, 0} \subseteq \dots \subseteq B_{i, w} \subseteq B_i$ such that
    \begin{enumerate}
        \item $|B_{i-1}| = 3^{-10\Delta w}|I_{i-1}| = b_{i-1}$,
        \item $b_i = |B_i| \geq |B_{i, j}| \geq 3^{-6\Delta w} |B_i| = m_i$, 
        \item $(B_{i-1} \cup B_{i, j-1}, \emptyset)$ is a $(B_{i, j}, \Delta, s_i, (3^{8\Delta w} \frac{s_i}{b_i})^{2\Delta} \binom{|I_{i-1} \cup B_{i, j-1}|}{\Delta})$-out-community, for all $j \in [w]$.
    \end{enumerate}
    Moreover, since $|I \cup B_{i, j-1}| \leq 3a_i$ 
    , we get that
    \begin{align*}
        \binom{|I_{i-1} \cup B_{i, j-1}|}{\Delta} \leq (3a_i)^{\Delta} \le 3^{17\Delta^2 w} m_i^{\Delta}
    \end{align*}
    and consequently,
    \begin{align*}
        \left(3^{8\Delta w} \frac{s_i}{b_i}\right)^{2\Delta} \binom{|I_{i-1} \cup B_{i, j-1}|}{\Delta} &\leq 2^{\Delta} \left(\frac{s_i}{2b_i}\right)^\Delta \left(3^{16\Delta w} \frac{s_i}{\delta^{-\frac{1}{\Delta}}3^{36\Delta w} s_i}\right)^{\Delta} 3^{17\Delta^2 w}m_i^{\Delta}\\
        &\leq 3^{2\Delta w}\left(\frac{s_i}{2b_i}\right)^{\Delta} 3^{-3\Delta^2 w} \delta  (3\Delta)^{\Delta} \left(\frac{m_i}{3\Delta}\right)^{\Delta}\\
        &\leq \delta \left(\frac{s_i}{2b_i}\right)^{\Delta} \binom{m_i/3}{\Delta},
    \end{align*}
    where we plugged in our definitions of $b_i$ and $m_i$.
    In particular, for each $j \in [w]$, $(B_{i-1} \cup B_{i, j-1}, \emptyset)$ is a $(B_{i,j}, \Delta, s_i, \delta (\frac{s_i}{2b_i}) ^{\Delta} \binom{m_i/3}{\Delta})$-out-community.

    By repeatedly applying this argument we are thus able to find all $B_i$'s for $i = 1, \dots, h$ and all $B_{i', j}$'s for $i' = 2, \dots, h$ and $j = 0, \dots, w$.
    Now, to find $B_{1, 0} \subseteq \dots \subseteq B_{1, w} \subseteq B_1$ we can once again apply \cref{lemma:forward_step}, this time with $I_0 = \emptyset$.
    By the same computations as above we get that $(B_{1, j-1}, \emptyset)$ is a $(B_{1, j}, \Delta, s_1, \delta(\frac{s_1}{2b_1})^{\Delta} \binom{m_1/3}{\Delta})$-out-community for each $j = 1, \dots, w$.
\end{proof}

\subsection{Embedding into the structure}\label{subsection:embedding_into_structure}
In this section, we prove \cref{theorem:main_theorem} by showing that we can embed the given digraph $G$ with graded bandwidth $w$ and maximum degree $\Delta$ into the structure given by \cref{theorem:embedding_structure}.
We find this embedding of $G$ layer-by-layer. As such, before giving the proof of \cref{theorem:main_theorem} in \cref{subsection:proof_main_theorem}, we first state and prove \cref{lemma:forward_step}, which allows us to embed one layer of $G$ at a time.

\subsubsection{A single embedding step}\label{subsection:forward_step}
To find an embedding of the next layer of $G$ into our structure, we will use the Lov\'{a}sz local lemma \cite{lovasz_lemma}, whose statement we now recall.
\begin{lemma}\label{lemma:local_lemma}
    Let $A_1, \dots, A_n$ be events in an arbitrary probability space and let $H = ([n], E)$ be a graph such that for each $i \in [n]$ the event $A_i$ is mutually independent of the events $\{ A_j : (i,j) \notin E \}$.
    Suppose moreover that $0 \leq x_1, \dots, x_n < 1$ are real numbers such that for all $i \in [n]$ we have $\Pr[A_i] \leq x_i \prod_{(i,j) \in E} (1- x_j)$.
    Then
    \[
        \Pr[\bigwedge_{i=1}^{n} \Bar{A_i} ] \geq \prod_{i=1}^n (1 - x_i) > 0.
    \]
\end{lemma}

We are now ready to state and prove our embedding lemma, which will allow us to embed one layer of the graded partition $V_1, \dots, V_H$ of $G$ at a time in the following manner.
At a given step we have already embedded $V_1$ into $A_1$, $V_2$ into $A_2$ and so forth, all the way up to embedding $V_{i-1}$ into $A_{i-1}$.
We would now like to embed $V_i$ into $A_i$ --- which will be the sets $W_1$ and $A$ in \cref{lemma:forward_step}, respectively, and we shall now denote them by the latter names.
Since for each $u \in W_1$ we have already embedded all the in-neighbors of $u$, we however cannot embed $u$ in just any vertex in $A$; let $f(u) \subseteq A$ be the set of vertices where we can embed $u$, i.e., the common out-neighborhood in $A$ of the vertices we have chosen as the embedding of the in-neighbors of $u$.
By carefully performing the embedding at the previous steps, we will be able to guarantee that for each $u \in W_1$ the set $f(u)$ will still be large, say $|f(u)| \geq b$ for some $b \in \mathbb{N}$.

Since where we embed some $u \in W_1$ affects the possible choices for the embedding of all the out-neighbors $v$ of $u$ in the future steps, we will also have to be careful in this step.
Specifically, we will let $W_2$ be the set of all vertices that have an in-neighbor in $V_i$ and let $D$ be a bipartite graph with parts $W_1$ and $W_2$ with edges between each $u \in W_1$ and all its out-neighbors $v \in W_2$ in $G$.
Then, for each $v \in W_2$, we will let ${\cal{F}}_v \subseteq 2^{A}$ be the set of all possible bad embeddings of the in-neighbors $u \in W_1$ of $v$ into $A$.
In our case, these will be all the embeddings that do not leave much space to embed $v$ later on.
Now, \cref{lemma:forward_step} simply says that as long as the sets ${\cal{F}}_v$ are not too dense, we can indeed embed $W_1$ into $A$ such that each $u \in W_1$ lands at some $x \in f(u)$ while at the same time avoiding all of these bad placements ${\cal{F}}_v$ of the in-neighbors of each $v \in W_2$.

\begin{lemma}\label{lemma:forward_step}
    Let $D$ be a bipartite graph with vertex classes $W_1$ and $W_2$ and let $\Delta^+, \Delta^- \ge 1$ be such that every vertex in $W_1$ has degree at most $\Delta^+$ and every vertex in $W_2$ has degree at most $\Delta^-$. Let $a, b \in \mathbb{N}$ be integers such that $a \geq b \geq 32|W_1|$ and let $A$ be a set of size $|A| = a$.
    For each $v \in W_2$, let ${\cal{F}}_v \subseteq 2^A$ be a collection of sets with $|N_D(v)|$-density at most $\frac{1}{4\Delta^+ \Delta^-} (2^{-1/2} \frac{b}{a})^{|N_D(v)|}$.
    Moreover, let $f: W_1 \to 2^A$ be a function such that for each $u \in W_1$ we have $|f(u)| \geq b$.
    Then, there exists an injective function $\phi: W_1 \to A$ such that for each $u \in W_1$ we have $\phi(u) \in f(u)$ and for each $v \in W_2$ we have $\phi(N_D(v)) \notin {\cal{F}}_v$.
\end{lemma}
\begin{proof}
    Let $s = |W_1|$ and let $\phi: W_1 \to A$ be a random mapping such that for each $u \in W_1$ the value of $\phi(u)$ is picked uniformly at random from the set $f(u)$, with all these choices made independently.
    For distinct $u, w \in W_1$ let $A_{uw}$ be the event that $\phi(u) = \phi(v)$.
    Moreover, for each $v \in W_2$ let us write $N_v = N_D(v)$ and let $B_v$ be the event that $|\phi(N_v)| = |N_v|$ but $\phi(N_v) \in {\cal{F}}_v$.
    Clearly, if none of the bad events $A_{uw}$ and $B_v$ hold, then $\phi$ satisfies the requirements of the lemma.

    Let us now bound the probabilities for each of these events to hold.
    For $u, w \in W_1$, we have $\Pr[A_{uw}] \leq \frac{1}{b}$.
    To bound $\Pr[B_v]$, consider a tuple $(\phi(u))_{u \in N_v}$ such that $|\phi(N_v)| = |N_v|$ but $\phi(N_v) \in {\cal{F}}_v$.
    Since the $|N_v|$-density of ${\cal{F}}_v$ is at most $\delta_v = \frac{1}{4\Delta^+ \Delta^-} (2^{-1/2} \frac{b}{a})^{|N_v|}$, the number of such tuples is at most
    \[
        \delta_v \binom{a}{|N_v|} |N_v|! \leq \delta_v a^{|N_v|}.
    \]
    Since the tuple $(\phi(u))_{u \in N_v}$ is chosen uniformly at random from a subset of $A^{|N_v|}$ of size at least $b^{|N_v|}$, this implies that $\Pr[B_v] \leq \frac{\delta_v a^{|N_v|}}{b^{|N_v|}} \leq \frac{1}{4\Delta^+ \Delta^-} 2^{-|N_v| / 2}$.

    Let us now consider the dependencies between the bad events.
    Note that since the random variables $\{\phi(u) \}_{u \in W_1}$ are mutually independent, the event $A_{uw}$ is mutually independent from all $A_{u' w'}$ and $B_v$ such that $\{u, w\} \cap \{ u', w'\} = \emptyset$ and $\{ u, w\} \cap N_v = \emptyset$, respectively.
    Thus, $A_{uw}$ is dependent on at most $2(s-2) < 2s$ events $A_{u'w'}$ and at most $2\Delta^+$ events $B_v$, where we recall that $s=|W_1|$.
    Similarly, the event $B_v$ is mutually independent from all $A_{uw}$ and $B_{v'}$ such that $N_v \cap \{ u, w\} = \emptyset$ and $N_v \cap N_{v'} = \emptyset$, respectively, and thus it is dependent on at most $s|N_v|$ events $A_{uw}$ and at most $|N_v| \Delta^+$ events $B_{v'}$.

    We now want to apply the local lemma.
    For each $A_{uw}$ we let the corresponding $x_i$ be $x = \frac{8}{b}$ and for each $B_v$ we let the corresponding $x_i$ be $y = \frac{1}{2\Delta^+ \Delta^-}$.
    Since $\Delta^- \geq 1$ and $b \geq 32s$ we then get
    \[
        x(1-x)^{2s}(1-y)^{2\Delta^+} \geq \frac{8}{b} 4^{-16s / b} 4^{-1 / \Delta^-} \geq \frac{1}{b} \geq \Pr[A_{uw}],
    \]
    where we used the inequality $1-z \geq 4^{-z}$, valid for all $0 \leq z \leq \frac{1}{2}$.
    Similarly, for all $v \in V_2$ we have
    \[
    y(1-y)^{|N_v| \Delta^+} (1-x)^{s |N_v|} \geq \frac{1}{2\Delta^+ \Delta^-} 4^{-|N_v| / (2\Delta^-)} 4^{-8s|N_v| / b} \geq \frac{1}{4\Delta^+ \Delta^-} 4^{-|N_v|/4} \geq \Pr[B_v],
    \]
    where we used that $|N_v| \le \Delta^-$.
    Therefore, by Lemma \ref{lemma:local_lemma}, the probability that none of the bad events $A_{uw}$ and $B_v$ occur is positive and thus, there exists a choice of $\phi$ satisfying the desired properties.
\end{proof}
\subsubsection{Proof of Theorem \ref{theorem:main_theorem}}\label{subsection:proof_main_theorem}
With \cref{lemma:forward_step} in hand, we are finally ready to prove \cref{theorem:main_theorem}.
As described above, we start by renaming the structure given by \cref{theorem:embedding_structure} to get sets $A_1, \dots, A_H$ such that for each $i \in [H]$ almost all $\Delta$-tuples in $A_{i - w} \cup \dots \cup A_{i - 1}$ have a large common out-neighborhood in $A_i$.
Then, we embed the layers of the graded partition $V_1, \dots, V_H$ of $G$ into the respective $A_i$'s one by one using \cref{lemma:forward_step}.
In particular, at each step after having embedded $V_1$ up to $V_{i-1}$, we will want to embed $V_i$ into $A_i$, while making sure that we will be able to continue embedding $V_{i+1}$ etc. in the future.
In particular, for each $v \in V_j$ with $j > i$, if $u_1, \dots, u_{\Delta}$ are the in-neighbors of $v$ in $G$ such that for some $\ell_1 < \ell_2 \leq \Delta$ we have already embedded $u_1, \dots, u_{\ell_1}$ and $u_{\ell_1 + 1}, \dots, u_{\ell_2} \in V_i$, we want to pick the embedding of $u_{\ell_1 + 1}, \dots, u_{\ell_2}$ in a way such that still for almost all choices of $\phi(u_{\ell_2 + 1}), \dots, \phi(u_{\Delta})$, we have that $\phi(u_1), \dots, \phi(u_{\Delta})$ have a large common out-neighborhood in $A_j$, which will give us a lot of space to put $v$.
Therefore, in the notation of \cref{lemma:forward_step}, for each such $v$, we will let ${\cal{F}}_v$ be the set of all $\phi(u_{\ell_1 + 1}), \dots, \phi(u_{\ell_2})$ that fail to satisfy this property.

\begin{proof}[Proof of \cref{theorem:main_theorem}]
    Let $G$ be a digraph with maximum degree at most $\Delta$ and graded bandwidth $w$, and let $V_1, \dots, V_H$ be its graded partition.
    Throughout, we will write $V_i = \emptyset$ for $i \notin [H]$.
    We let $h$ be the smallest integer such that $hw \geq H$ and for each $i \in [h]$ define
    \[
        N_i = \max \{ |V_{iw - (w-1)}|, |V_{iw - (w-2)}|, \dots, |V_{iw}|\}.
    \]
    Moreover, we let
    \[
        n_i' = \sum_{j=1}^h \frac{N_j}{2^{|i -j|}}
    \]
    and note that $n_i' \geq \max \{ \frac{n'_{i-1}}{2}, \frac{n'_{i+1}}{2}, N_i \}$. We also observe for future reference that 
    \[
    \sum_{i=1}^h n_i' =\sum_{i=1}^h \sum_{j=1}^h \frac{N_j}{2^{|i-j|}} = \sum_{j=1}^h N_j \sum_{i=1}^h \frac{1}{2^{|i-j|}}
    \leq \sum_{j=1}^h 4N_j.
    \]
    We define $s_i' = 64wn'_i$, set $\delta = (\frac{1}{4 \Delta^2})^w 2^{-\Delta/2}$, and let $a_i'$, $b_i'$ and $m_i'$ 
    be defined as in \cref{theorem:embedding_structure}.
    Moreover, for $0 \leq k \leq w$ and $0 \leq d \leq \Delta$, we set $\delta_{k, d} = (\frac{1}{4\Delta^2})^k \cdot (2^{-1/2}\frac{s_1'}{2b_1'})^{d}$. 
    Note that, in particular, $\delta_{w,\Delta} = \delta_{\ref{theorem:embedding_structure}} (\frac{s_1'}{2b_1'})^\Delta$.

    We want to show that $\vv{r}(G) \leq 3^{57\Delta w}|V(G)|$.
    Therefore, let $T$ be an arbitrary tournament on at least
    \begin{align*}
        3^{57\Delta w}|V(G)| &\geq 3^{54\Delta w }w \cdot \sum_{i=1}^H 4 |V_i| \geq 3^{54\Delta w} \sum_{i=1}^h 4 N_i \\
        &\geq 3^{54\Delta w} \sum_{i=1}^h n_i' \geq 3^{52\Delta w} \sum_{i=1}^h \delta_{\ref{theorem:embedding_structure}}^{-1/\Delta}n_i'\\
        &\geq \sum_{i=1}^h 3^{48\Delta w} \cdot \delta_{\ref{theorem:embedding_structure}}^{-1/\Delta} \cdot 64w n_i' \geq \sum_{i=1}^h 6a_i'
    \end{align*}
    vertices.
    By \cref{theorem:embedding_structure} we can find disjoint $B_1, \dots, B_h \subseteq V(T)$ together with $B_{i, 0} \subseteq \dots \subseteq B_{i, w} \subseteq B_i$ for each $i \in [h]$ such that for each $i \in [h]$ and $j \in [w]$ we have
    \begin{enumerate}
        \item $m_i' \leq |B_{i, j}| \leq b_i'$, \emph{and}
        \item $(B_{i-1} \cup B_{i, j-1}, \emptyset)$ is a $(B_{i, j}, \Delta, s'_i, \delta_{w, \Delta} \binom{m_i' / 3}{\Delta})$-out-community.
    \end{enumerate}
    We aim to show that we can embed $G$ into this structure, which will prove $\vv{r}(G) \leq 3^{57\Delta w} |V(G)|$.

    To that end, we first rename the $B_{i, j}$'s to make it easier to formulate our embedding strategy.
    For each $i \in [H]$ we let
    \[
        A_i = B_{\lceil i/w \rceil, ((i -1) \text{ mod } w) + 1}
    \]
    and let $A_i = \emptyset$ for all $i  <1$.
    Similarly, once again for all $i \in [H]$, we let
    \[
        n_i = n_{\lceil i/ w \rceil}', \qquad s_i = s_{\lceil i/ w \rceil}', \qquad m_i = m_{\lceil i/ w \rceil}' \qquad \text{and} \qquad b_i = b_{\lceil i/ w \rceil}'.
    \]
    We note that by the two properties of the sets $B_i$ and $B_{i, j}$ we immediately get that for each $i \in [H]$
    \begin{enumerate}
        \item $m_i \leq |A_i| \leq b_i$, \emph{and}
        \item $(A_{i-w} \cup \dots \cup A_{i-1}, \emptyset)$ is an $(A_i, \Delta, s_i, \delta_{w, \Delta}\binom{m_i/3}{\Delta})$-out-community.
    \end{enumerate}
    Indeed, for the second property notice that if $i = aw + b$ for some $0 \leq a \leq h$ and $0 \leq b < w$, i.e., $A_i = B_{a, b+1}$, then $A_{i-w} \cup \dots \cup A_{i - 1} \subseteq B_{a-1} \cup B_{a, b}$.
    For future reference, we also have that $|A_{i-j}| \geq m_i / 2 \geq m_i/3 + \Delta$ for each $j \in [w]$.

    We now want to embed the $V_i$'s into the corresponding $A_i$'s; we do it inductively, using \cref{lemma:forward_step} and starting with embedding $V_1$ into $A_1$.
    More specifically, we let $P_i = \bigcup_{j=1}^i V_i$ be the set of vertices of $G$ we have already embedded up to step $i \in [H]$ and, for each $v \in V(G)$, we let $P_i(v) = N^-_G(v) \cap P_i$ denote the set of in-neighbors of $v$ that we have already embedded at that step.
    Moreover, we let $r_i(v) = |N^-_G(v) \setminus P_i(v)|$ denote the number of in-neighbors of $v$ that still have not been embedded up to step $i$.
    We also define \( k_{i,j} = \min \{ w, j - i - 1 \} \) for each \( 0 \leq i < j \leq H \).  
    In other words, at each step $i$, for each unembedded layer $V_j$, $k_{i, j}$ represents the number of unembedded layers that could potentially send an edge to $V_j$.  
    
    For each $i = 0, \dots, H$, we will find a partial embedding $\phi_i : P_i \to \bigcup_{j=1}^i A_i$ such that
    \begin{enumerate}
        \item $\phi_i$ is an embedding of $G[P_i]$ into $T[\bigcup_{j'=1}^i A_{j'}]$,
        \item $\phi_i(V_{j'}) \subseteq A_{j'}$ for each $j' \in [i]$, \emph{and}
        \item for each $i < j \leq H$ and each $v \in V_j$ we have that $(A_{j-k_{i,j}} \cup \dots \cup A_{j-1}, \phi_i(P_i(v)))$ is an \newline $(A_j, r_i(v), s_j, \delta_{k_{i, j}, r_i(v)}\binom{m_j/3}{r_i(v)})$-out-community.
    \end{enumerate}
    We note that condition 3 in particular implies that for each $i \in [H-1]$ and $v \in V_{i+1}$ we have $|N_T^+(\phi_i(V_i)) \cap A_{i+1}| \geq s_{i+1}$, i.e., there are at least $s_{i+1}$ vertices in $A_{i+1}$ where we can embed $v$ to extend $\phi_i$.

    We first check that for $i = 0$ condition 3 indeed holds.
    To that end, fix $j \in [H]$ and $v \in V_j$.
    Note that $|P_0(v)| = 0$ and $r_0(v) \leq \Delta$ since $G$ has maximum degree at most $\Delta$.
    Let $E$ be the set of all $r_0(v)$-subsets $S\subseteq A_{j - k_{0, j}} \cup \dots \cup A_{j-1}$ such that $|N^+(S) \cap A_j| \leq s_j$.
    Moreover, let $C$ be the set of all $r_0(v)$-subsets $S \subseteq A_{j - k_{0, j}} \cup \dots \cup A_{j-1}$ such that $(A_{j - k_{0, j}} \cup \dots \cup A_{j-1}, S)$ is not an $(A_j, \Delta - r_0(v), s_j, \binom{m_j/3}{\Delta - r_0(v)})$-out-community.
    Clearly, $E \subseteq C$ and moreover, by \cref{observation:community_extension_density}, $|C| \leq \delta_{w, \Delta} \binom{m_j/3}{r_0(v)} \leq \delta_{k_{0, j}, r_0(v)} \binom{m_j/3}{r_0(v)}$.
    In particular, we get that $(A_{j - k_{0, j}} \cup \dots \cup A_{j-1}, \phi(P_0(v)))$ is an $(A_j, r_0(v), s_j, \delta_{k_{0, j}, r_0(v)} \binom{m_j/3}{r_0(v)})$-out-community as required. 

    Now, assume that for some $i \in [H]$ we have already found $\phi_{i-1}$ satisfying the above conditions.
    We want to extend it to $\phi_i$ using \cref{lemma:forward_step}.
    To that end, we define $N_i = \bigcup_{u \in V_i} N^+(u)$, and note that since $G$ has graded bandwidth $w$, we have $N_i \subseteq V_{i+1} \cup \dots \cup V_{i+w}$.
    We let $D$ be the bipartite graph obtained from $G[V_i \cup N_i]$ by removing the orientation of the edges.
    For each $u \in V_i$, we also let $f(u) = (N_T^+(\phi(P_{i-1}(u))) \cap A_i)\setminus \phi(P_{i-1})$ denote the set of vertices from $A_i$ we can embed $u$ into so that we get a valid extension of $\phi_{i-1}$.
    Notice that  $A_{i}$ is disjoint from all $A_{i'}$ with $i' < j-w$ and that for all $j-w \leq i' \leq i $ we have $|V_{i'}| \leq 2n_i$.
    Therefore, by condition 3 for $\phi_{i-1}$, we get that $|f(u)| \geq 64wn_i - 2wn_i \geq 32wn_i \geq 32|V_i|$ for each $u \in V_i$.

    For each $i+ 1\leq j \leq i + w$ and each $v \in V_{j}$ we now set ${\cal{F}}_v$ as the set of all $S \subseteq A_i$ such that $(A_{i+1} \cup \dots \cup A_{j-1}, \phi(P_{i-1}(v))\cup S)$ is not an $(A_{j}, \Delta, s_{j}, \delta_{k_{i, j}, r_i(v)} \binom{m_{j}/3}{\Delta})$-out-community.
    Note that $|N_D(v)| = r_{i-1}(v) - r_i(v).$ Thus, since $|A_i \setminus \phi(P_i(v))| \geq |A_i| - \Delta \geq m_{j}/3$, by \cref{observation:community_extension_density} we get that ${\cal{F}}_v$ is upward closed and has $|N_D(v)|$-density at most
    \[ \frac{\delta_{k_{i-1,j}, r_{i-1}(v)} \binom{m_j/3}{r_{i-1}(v)}}{\delta_{k_{i,j}, r_i(v)}} = \frac{1}{4\Delta^2} \left(2^{-1/2} \frac{s_1'}{2b_1'}\right)^{|N_D(v)|} \le \frac{1}{4\Delta^2}\left(2^{-1/2} \frac{32wn_i}{|A_i|}\right)^{|N_D(v)|}. \]
    
    

    Therefore, by \cref{lemma:forward_step} we can find an injection $\phi: V_i \to A_i$ such that $\phi(u) \in f(u)$ for each $u \in V_i$ and $\phi(N_D(v)) \notin {\cal{F}}_v$ for each $v \in N_i$.
    In particular, by our choice of $f$ and the ${\cal{F}}_v$'s, the embedding $\phi_i: P_i \to \bigcup_{j=1}^i A_i$ defined by
    \[
        \phi_i(u) = \begin{cases}
            \phi_{i-1}(u), & u \in P_{i-1}\\
            \phi(u), & u \in V_i
        \end{cases}
    \]
    satisfies all of the above conditions.
    At the end of this process, we have constructed an embedding $\phi_H:P_H \to T$. Since $P_H=G$, we conclude that $G \subseteq T$, concluding the proof.
\end{proof}

\section{Upper bound for graded digraphs}\label{section:graded_upper_bound}
By plugging in $w=1$ to \cref{theorem:main_theorem}, we immediately get that $\vv{r}(G) \leq 3^{57\Delta}|V(G)|$ for any graded digraph with maximum degree $\Delta$.
However, for this special case, the problem is actually much simpler than for digraphs with graded bandwidth $w \geq 2$, which is why we are able to prove the stronger bound stated in \cref{theorem:easy_upper_bound}.
In particular, finding sets $A_1, \dots, A_H$ in the host tournament such that for each $i \in [H-1]$ almost all $\Delta$-subsets of $A_{i}$ have many common out-neighbors in $A_i$ is possible without the detour through the sets $B_i$ and $B_{i, j}$, which was needed for the general case.
In fact, having found $A_{i+1}, \dots, A_H$ --- again making sure that they are placed correctly in the median order of the host tournament --- we are able to find $A_i$ with just a single application of dependent random choice.
Another observation --- which this time also applies for the general case as well --- is that in the proof of \cref{theorem:main_theorem} the role of the in- and out-degree are actually asymmetric.
These two observations allow us to replace $\Delta$ in the exponent with the maximum in-degree $\Delta^-$ and significantly improve the constant in front of it.

As stated before, instead of proving \cref{theorem:easy_upper_bound}, we will prove a stronger theorem that also leverages the fact that the graded digraph could locally have different maximum in-degrees in different parts.
This will allow us to use less overhead to embed the parts where the in-degree is small, and is particularly useful for graded digraphs in which the large in-degree only appears in parts of the graded partition whose sizes are small, as is the case for the oriented hypercube $\vv{Q_d}$.
More formally, we can prove the following theorem, which immediately implies \cref{theorem:easy_upper_bound} by upper-bounding each $\Delta_i^-$ by $\Delta^-$.
\begin{theorem}\label{theorem:graded_upper_bound}
    Let $G$ be a graded digraph with a graded partition $V(G) = V_1 \cup \dots \cup V_H$ for some $H \in \mathbb{N}$ and maximum in- and out-degree $\Delta^-$ and $\Delta^+$ respectively. 
    Moreover, for each $i\in [H-1]$ let $\Delta^-_i$ be the maximum in-degree in the induced subgraph $G[V_i \cup V_{i+1}]$ and set $\Delta^-_0 = \Delta^-_{H} = 0$.
    Then $$\vv{r}(D) \leq 10^9 (\Delta^-)^2 \Delta^+ \sum_{i=1}^H 2^{2(\Delta_{i-1}^- + \Delta_{i}^-) } |V_i|.$$
\end{theorem}
We defer the proof of \cref{theorem:graded_upper_bound} to \cref{section:appendix}, due to its similarity to the proof of \cref{theorem:main_theorem}.
We also note that \cref{theorem:graded_upper_bound} immediately implies the upper bound on the oriented Ramsey number of the hypercube.
\begin{proof}[Proof of  \cref{thm:hypercube}]
    Let $d \in \mathbb{N}$ and $\vv{Q_d}$ be the $d$-dimensional hypercube on the vertex set $V = \{ 0,1 \}^d$.
    Moreover, for each $0\leq i \leq d$, let $V_i := \{ v\in V : \sum_j v_j = i \}$.
    Then $V = V_0 \cup \dots \cup V_d$ is a graded partition of $\vv{Q_d}$ and for each $i = 0, \dots, d-1$ the maximum in-degree in the induced subgraph $\vv{Q_d}[V_i \cup V_{i+1}]$ is $i+1$.
    Therefore, by Theorem \ref{theorem:graded_upper_bound},
    \[
        \vv{r}(\vv{Q_d}) \leq C' d^3 \sum_{i=0}^d \binom{d}{i} 4 \cdot 2^{4i} = C d^3 (16 + 1)^d = C d^3 17^d,
    \]
    for some absolute constants $C'$ and $C$.
\end{proof}

\section{Lower bound for graded digraphs}\label{section:lower_bound}
\subsection{Proof outline}
In this section, we prove \cref{theorem:lower_bound}, which states that there exists a graded digraph $G$ with height $h$, maximum degree $\Delta$, and equal number of vertices in each part of the graded partition, such that $\vv r(G) \geq c^\Delta \ab{V(G)}$ for an absolute constant $c>1$.

The main ingredient of the proof is the same statement in the case $h=2$, from which the general statement will follow.
In other words, we first show that exists a bipartite digraph $D_0$ with vertex classes of size $n$ each and maximum degree at most $\Delta$ such that $\vv{r}(D_0) \geq c^\Delta \cdot(2n)$.

Having found such a bipartite digraph, we will generalize the construction for any height $h$, by taking a graded digraph $G$ such that the induced subgraph between two neighboring parts of the graded partition is a copy of $D_0$.
We will show that such a $G$ is not contained in a tournament $T$ obtained by replacing each vertex of a transitive tournament on $H = h/2 - 1$ vertices with a copy of $R$, a large tournament not containing $D_0$.

Indeed, if we let $G$ have the graded partition $V(G) = V_1 \cup \dots \cup V_h$ and let $V(T) = A_1 \cup \dots \cup A_{H}$ such that each $T[A_i]$ is a copy of $R$ and all the edges go from $A_i$ to $A_j$ for $i < j$, then we show the following.
If an embedding of $G$ into $T$ exists, then for any $i$, if we embedded most of $V_i$ into $A_j \cup \dots \cup A_H$ for some $j$, then most of $V_{i+2}$ must be embedded into $A_{j+1} \cup \dots \cup A_H$.
This will give us a contradiction since there are only $H < h/2$ levels in the tournament $T$.

It remains to construct $D_0$ and $R$ such that the argument above works. We will use a construction very similar to the one of Graham, Rödl and Ruciński \cite{MR1832445}. Namely, we show that if we take $R$ to be a blow-up of a random tournament and $D_0$ a sparse random bipartite digraph, then with positive probability $R$ does not contain a copy of $D_0$.
To make the generalization for any height possible, we will in fact need to show a slightly stronger statement, namely that if we take any two large subsets $A'$ and $B'$ of the two vertex classes of $D_0$, then $R$ does not contain a copy of the induced subgraph $D_0[A' \cup B']$.

The remainder of this section contains the details of the argument. We first construct a suitable bipartite digraph $D_0$ and a suitable tournament $R$ in \cref{section:GRR_lemmas}; these constructions are encapsulated in two technical lemmas, which are very similar to ones appearing in \cite{MR1832445}.
Then in  \cref{section:bipartite_case} we show that $R$ indeed does not contain a copy of $D$.
Finally, in \cref{section:proof_lower_bound} we give a proof of \cref{theorem:lower_bound} by generalizing the construction to all heights.
\begin{remark}
    As in \cite{MR1832445}, we quantify the respective sizes of the sets and other quantities with concrete numerical values.
    For example, ``large'' subsets of the vertex classes $A$ and $B$ of $D_0$ will mean $A' \subseteq A$ and $B' \subseteq B$ of sizes at least $0.98|A|$ and $0.98|B|$, respectively.
    However, we want to stress that the actual numerical values are not of much importance; what matters is that the dependencies between them work out correctly.  
\end{remark}

\subsection{Technical lemmas for the bipartite case}\label{section:GRR_lemmas}
We begin by constructing the bipartite digraph $D_0$, which we will require to have bounded degree and satisfy certain pseudorandom properties. The following lemma is a slight generalization of \cite[Lemma 3]{MR1832445}, and states that a graph with such properties can be constructed randomly. We use the notation $e_H(X,Y)$ to denote the number of pairs in $X \times Y$ that are edges of $H$.
\begin{lemma}\label{lemma:guest_graph}
    There exist constants $c_0 > c_1 > 1$ and $\Delta_0$ such that for each $\Delta \geq \Delta_0$ and $n \geq (c_0)^{2\Delta}$ there exists a bipartite graph $H$ with vertex classes $X$ and $Y$ of size $n$ and maximum degree at most $\Delta$ such that the following hold, where $k=(c_0)^{\Delta}$.
    \begin{enumerate}
        \item For all partitions $X = X_1 \cup \dots \cup X_k \cup D_X$ and $Y = Y_1 \cup \dots \cup Y_k \cup D_Y$ with $|X_i|, |Y_i| \leq (c_1 / c_0)^\Delta n$ and $|D_X|, |D_Y| \leq 0.02n$, we have
        \[
            \sum_{i \neq j: e_H(X_i, Y_i) > 0} |X_i||Y_j| > 0.55(0.98n)^2.
        \]
        \item For all $X' \subseteq X$ and $Y' \subseteq Y$ such that $|X'|, |Y'| \geq 0.01n$, we have $e_H(X', Y') > 0$.
    \end{enumerate}
\end{lemma}
The proof of \cref{lemma:guest_graph} is a standard union bound argument, and we include it in \cref{appendix:GRR} for completeness.

Our next lemma is of a similar flavor to \cref{lemma:guest_graph}, showing that a random object typically satisfies a certain pseudorandom property, and is very similar to \cite[Lemma 4]{MR1832445}. In this case, we show that a random tournament typically has many directed edges from any large set to any other large set. We actually prove something slightly more general, which says the same not for sets, but for ``weighted sets'', that is, for functions valued in $[0,1]$. Here, and in the rest of the proof, all logarithms are to base $e$.
\begin{lemma}\label{lemma:host_graph}
    Let $k \geq 2$ and $x \geq (10^8 \log k) /2$. There exists a tournament $R$ with vertex set $[k]$ such that for all pairs of weight functions $f, g: [k] \to [0,1]$ with $f+g \leq 1$ and $\sum_{i=1}^k (f(i) + g(i)) = 2x$, we have
    \[
        W \coloneqq \sum_{ij \in E(R)} f(i)g(j) \leq 0.51x^2.
    \]
\end{lemma}
Again, the proof of \cref{lemma:host_graph} is fairly standard, and we defer it to \cref{appendix:GRR}.

\subsection{The bipartite case}\label{section:bipartite_case}
We are now ready to prove \cref{theorem:lower_bound} in the case $h=2$, that is, when $G$ is bipartite. As discussed above, this step is actually the heart of the proof, as the proof for arbitrary $h$ will essentially be a reduction to this case.
\begin{lemma}\label{lemma:bipartite_lower_bound}
    There exist constants $c' > 1$ and $\Delta_0$ such that for all $\Delta \geq \Delta_0$ and $n \geq \Delta $ the following holds. 
    There exists a bipartite digraph $D_0 = (A \cup B, E)$ with $|A| = |B| \leq n$ and maximum degree $\Delta$  such that all its edges are directed from $A$ to $B$, as well as a tournament $R$ on $(c')^\Delta n$ vertices such that for any $A' \subseteq A$ and $B' \subseteq B$, the following hold.
    \begin{enumerate}
        \item If $|A'|, |B'| \geq 0.98|A|$ then there is no copy of $D_0[A' \cup B']$ in $R$, {and}
        \item if $|A'|, |B'| \geq 0.01|A|$ then $e_{D_0}(A', B') > 0$.
    \end{enumerate}
\end{lemma}
\begin{proof}
    Let $c_0, c_1$ and $\Delta_0$ be the constants from \cref{lemma:guest_graph}. By potentially increasing $\Delta_0$ further, we may also assume that $(c_0 / c_1)^{\Delta} > 10^9 \Delta \log c_0$ for all $\Delta \geq \Delta_0$.
    We let $1 < c' = \min\{ c_1, 2^{0.3}/c_0\}$.

    If $\Delta \geq \Delta_0$ and $\Delta  \leq n < \frac{1}{0.98}2c_0^\Delta$, let $D_0$ be the oriented complete bipartite graph $\vv{K}_{\Delta ,  \Delta }$ with vertex classes $A$ and $B$ and where all the edges are oriented from $A$ to $B$.
    Clearly, for any non-empty $A' \subseteq A$ and $B \subseteq B'$ we have $e_D(A', B') > 0$.
    To show that the first property also holds, note that if $|A'|, |B'| > 0.98 \Delta $ then $D_0[A' \cup B']$ has the complete bipartite graph $\vv{K}_{\lfloor0.98\Delta\rfloor, \lfloor0.98\Delta\rfloor}$ as a subgraph.
    Therefore, since the probability that a uniformly random tournament $R$ on $N = 2^{\lfloor0.98\Delta\rfloor/2} \geq (c')^\Delta n$ vertices contains a copy of $\vv{K}_{\lfloor0.49\Delta\rfloor, \lfloor0.49\Delta\rfloor}$ is at most $\binom{N}{\lfloor0.98\Delta\rfloor}\binom{N}{\lfloor0.98\Delta\rfloor} 2^{-\lfloor0.98\Delta\rfloor^2} < N^{2\lfloor0.98\Delta\rfloor} 2^{-\lfloor0.98\Delta\rfloor^2} \leq 1$, there is a choice of $R$ such that the conditions of the lemma are satisfied.

    Otherwise, we have $0.98n \geq 2c_0^\Delta$.
    Let then $D_0$ be a digraph obtained by taking the graph $H$ from Lemma \ref{lemma:guest_graph} and orienting every edge from $A$ to $B$, and let $R'$ be the tournament from Lemma \ref{lemma:host_graph}.
    We obtain a tournament $R$ by taking the blow-up of $R'$ in the following way.
    Let $N = c_1^\Delta n \geq (c')^\Delta n$, $k = c_0^\Delta$ and partition $[N] = U_1 \cup \dots \cup U_k$ such that $|U_i| = N/k$. 
    Let $R$ be an arbitrary tournament on the vertex set $[N]$ such that for all $ij \in E(R')$ we have $U_i \times U_j \subseteq E(R)$. That is, the edges between distinct $U_i,U_j$ form a blown-up copy of $R'$, and the edges inside any $U_i$ are oriented arbitrarily.
    
    $D_0$ satisfies the second condition of the lemma by \cref{lemma:guest_graph}.
    Suppose now for contradiction that for some $A' \subseteq A$ and $ B' \subseteq B$ with $|A'|, |B'| \geq 0.98n$ there is a copy of $D_0[A' \cup B']$ in $R$ and let $X$ and $Y$ be its two vertex classes.
    By Lemma \ref{lemma:guest_graph} we have that for $X_i \coloneqq X \cap U_i$ and $Y_i \coloneqq Y \cap U_i$ it holds
    \[
        \sum_{ij \in E(R')} |X_i||Y_j| \geq \sum_{i \neq j: e_{D}(X_i, X_j) > 0} |X_i||Y_j| > 0.55(0.98n)^2.
    \]
    For $i \in [k]$, let $f(i) = \frac{|X_i|k}{N}$ and $g(i) = \frac{|Y_i|k}{N}$. We have $0 \leq f+g \leq 1$ and
    \[
        2x \coloneqq \sum_i (f(i) + g(i)) = \frac k N(\ab X + \ab Y)\geq  2 \cdot 0.98 \frac{nk}{N} = 1.96\left( \frac{c_0}{c_1} \right)^{\Delta} \geq 10^9 \Delta \log c_0 > 10^8 \log k.
    \]
    Therefore, by \cref{lemma:host_graph},
    \[
        \sum_{ij \in E(R')} |X_i||Y_j| = \frac{N^2}{k^2} \sum_{ij \in E(R)}f(i)g(j) < \frac{N^2}{k^2} 0.51x^2 \leq 0.51 (0.98n)^2,
    \]
    a contradiction. This shows that there is no copy of $D_0[A' \cup B']$ in $R$ for any such $A'$ and $B'$. 
\end{proof}

\subsection{Proof of Theorem \ref{theorem:lower_bound}}\label{section:proof_lower_bound}
With the ingredients above, we are ready to prove \cref{theorem:lower_bound}. 
\begin{proof}[Proof of Theorem \ref{theorem:lower_bound}]
    Let $c'>1$ and $\Delta'_0$ be the constants from Theorem \ref{lemma:bipartite_lower_bound} and set $\Delta_0$ to be a constant such that $\Delta_0 \geq \Delta_0'$ and $(c')^{\Delta_0/2} > 4$.
    Moreover, let $c>1$ be a constant such that $c^{\Delta_0} = (c')^{\Delta_0/2}/4$.
    Let $\Delta \geq 2\Delta_0$ and $n \geq \Delta $.
    Let $D_0 = (A \cup B, E)$ and $R$ be respectively the bipartite digraph and tournament from Lemma \ref{lemma:bipartite_lower_bound}, applied with the parameters $n$ and $\Delta/2$.

    If $h=2$, then we take $G = D_0$ and since $R$ doesn't contain a copy of $D_0$ we get $\vv{r}(G) \geq (c')^{\Delta/2} n \geq c^\Delta n$.
    Otherwise, we define a graded digraph $G$ on $nh$ vertices and with a graded partition $V_1 \cup \dots \cup V_h$, where $\ab{V_i}=n$ for all $i$, by declaring that
    for all $i \in [h-1]$, the induced subgraph $G[V_i \cup V_{i+1}]$ is a copy of $D_0$ such that $V_i$ plays the role of $A$ and $V_{i+1}$ plays the role of $B$. Note that the maximum degree in $G$ is at most $\Delta$, since the maximum in-degree and maximum out-degree are both at most $\Delta/2$.

    Now let $H = \lceil\frac{h}{2}\rceil -1  > 1$ and define a tournament $T$ on $(c')^{\Delta/2} Hn$ vertices with a vertex partition $V(T) = A_1 \cup \dots \cup A_H$, where
    $|A_i| = (c')^{\Delta/2} n$ for each $i \in [H]$, as follows. For each $i$, we let
    $T[A_i]$ be a copy of $R$, and for all $1 \leq i<j\leq H$, we direct all edges from $A_i$ to $A_j$. 
    We claim that there is no copy of $G$ in $T$.
    
    Indeed, suppose for contradiction that $\phi$ is an embedding of $G$ into $T$.
    For $i \in [h]$ and $A' \subseteq V(T)$, let $f_i(A') = \frac{|\phi(V_i) \cap A'|}{|V_i|}$ be the fraction of vertices of $V_i$ embedded into $A'$. Additionally, let $U_i = \bigcup_{j=i}^H A_j$, and let $j_i$ be the largest index $1 \leq j \leq H$ such that $f_i(U_j) \geq 0.99$. Note that this is well-defined since $U_1=V(T)$, and hence $f_i(U_1)=1$. 

    The key observation is that if $f_i(U_j)\geq 0.01$, then $f_{i+1}(U_j) \geq 0.99$. Indeed, if this is not the case, then there exist $X_i \subseteq V_i, X_{i+1} \subseteq V_{i+1}$ with $\ab{X_i},\ab{X_{i+1}}\geq 0.01 n$, with the properties that $\phi(X_i) \subseteq U_j$ and $\phi(X_{i+1}) \subseteq V(T) \setminus U_j$. But all edges in $T$ are directed from $V(T) \setminus U_j$ to $U_j$, hence the second condition in \cref{lemma:bipartite_lower_bound} implies that if such $X_i,X_{i+1}$ exist, then $\phi$ is not a valid embedding.

    In particular, applying this observation with $j=j_i$, we conclude that $f_{i+1}(U_{j_i}) \geq 0.99$. This implies that $j_{i+1}\geq j_i$ for all $i \in [h-1]$, that is, that the indices $j_i$ are monotonically non-decreasing. 
    We now claim that for each $i \in [h-2]$, we have that $j_{i+2}>j_i$.

    Indeed, if $j_{i+1}>j_i$, then we are done by the monotonicity property $j_{i+2}\geq j_{i+1}$. Hence we may assume that $j_i=j_{i+1}$, which in particular implies that $f_i(U_{j_i+1}) < 0.01$ by the key observation. If $f_{i+1}(U_{j_i+1})\geq 0.01$, then we are again done by the key observation. Therefore, we may assume that $f_{i+1}(U_{j_i+1})<0.01$. Together with the fact that $j_{i+1}=j_i$, we conclude that $f_i(A_{j_i}), f_{i+1}(A_{j_i})\geq 0.98$. 

    In other words, there exist $X_i \subseteq V_i, X_{i+1} \subseteq V_{i+1}$ with $\ab{X_i},\ab{X_{i+1}} \geq 0.98n$ such that $\phi(X_i), \phi(X_{i+1}) \subseteq A_{j_i}$. But this is a contradiction to \cref{lemma:bipartite_lower_bound}, since $T[A_{j_i}]$ is a copy of $R'$. We conclude that, as claimed, $j_{i+2}> j_i$ for all $i \in [h-2]$.

    Since $j_1 \geq 1$ and $j_{i+2} \geq j_i + 1$ for all $i$, we find that $j_i \geq i/2$ for all $i$. In particular, $j_h \geq h/2 > H$. But this is a contradiction as there are only $H$ parts in $T$, implying that there is no copy of $G$ in $T$.
    Therefore, 
    \[
        \vv{r}(G) >  (c')^{\Delta/2} Hn \geq \frac 14 (c')^{\Delta/2} hn \geq c^\Delta hn.\qedhere
    \]
\end{proof}

\section{Concluding remarks}\label{section:concluding_remarks}
While \cref{theorem:easy_upper_bound} is roughly best possible in general, it is reasonable to expect that one could improve it in certain cases. In particular, for the oriented hypercube $\vv{Q_d}$, we expect that the bound in \cref{thm:hypercube} could be significantly improved.
In fact, our techniques are already sufficient show that the induced subgraph of $\vv{Q_d}$ obtained by taking the vertices with at most $d/2$ non-zero coordinates has oriented Ramsey number at most $2^{3d+o(d)}$. As this digraph consists of simply the first half of the graded partition of $\vv{Q_d}$, this suggests to us that the bound in \cref{thm:hypercube} is not particularly close to best possible.
Concretely, we make the following conjecture, which is a directed analogue of the
Burr--Erd\H{o}s conjecture \cite{MR371701} that $r(Q_d) = O(2^d)$.
\begin{conjecture}
    There is an absolute constant $C>0$ such that $\vv r(\vv{Q_d}) \leq C2^d$ for all $d \geq 1$.
\end{conjecture}

As mentioned in the introduction, we believe that \cref{theorem:main_theorem} gives the most general condition that is currently known for guaranteeing that a bounded-degree acyclic digraph has linear oriented Ramsey number. However, it is certainly far from being a full characterization of this property, and it would be very interesting to find additional structural notions that imply linear bounds on oriented Ramsey numbers.

Finally, we reiterate the main question left open from \cite{MR4819947}, which we consider to be of central importance in the study of oriented Ramsey numbers.
\begin{question}
    Given $\Delta \geq 1$, does there exist some $C>0$ such that every $n$-vertex acyclic digraph $D$ with maximum degree $\Delta$ satisfies $\vv r(D) \leq n^C$?
\end{question}

\paragraph{Acknowledgments:} We are grateful to Xioayu He for constructive comments on an earlier draft of this paper.

\appendix
\section{A more refined analysis for graded digraphs}\label{section:appendix}
In this appendix, we prove \cref{theorem:graded_upper_bound}.
As stated before, the proof follows the same lines as the proof of \cref{theorem:main_theorem}.
We first fix a suitably large host tournament $T$ and, again using dependent random choice and the median order, find disjoint sets $A_1, \dots, A_H \subseteq V(T)$ such that for each $i \in [H-1]$ almost all $\Delta_i^-$ subsets of $A_i$ have many common out-neighbors in $A_{i+1}$.
Then, using Lov\'asz local lemma we show that we can embed our graded digraph into this structure layer-by-layer.

As described above, in the case of graded digraphs, having found the sets $A_{i+1}, \dots, A_H$ just a single application of dependent random choice suffices to find the set $A_i$ --- which is formalized in the following lemma.
\begin{lemma}\label{lemma:graded_backward_step}
    Let $a, a', b, \ell, s, N, k, \Delta^- \in \mathbb{N}$ be integers with $2a' \geq a$, and let $T$ be a tournament on $N$ vertices with a median order $v_1, \dots, v_N$.
    Moreover, let $2\ell a' < o \leq N - a + 1$ and let $B \subseteq [o, o+a)$ be an arbitrary subset of size at least $b$.
    Then there exist $o - 2 \ell a' \leq o' \leq o-a'$ and $A \subseteq [o', o' + a')$ such that
    \begin{itemize}
        \item $|A| \geq \frac{a'}{2} (\frac{\ell-1}{2\ell})^k,$ {and}
        \item $(A, \emptyset)$ is a $(B, \Delta^-, s, 4(\frac{2\ell}{\ell-1})^k\binom{a'}{\Delta^-}(\frac{s}{b})^\ell)$-out-community.
    \end{itemize}
\end{lemma}
\begin{proof}
    Note that we may assume $\ab B = b$, because the conclusion for any larger value of $\ab B$ is a strictly weaker statement. 
    Let $J = [o-2\ell a', o)$ and since $B \subseteq [o, o+a)$, by \cref{observation:median_order} we have
    \[
        \sum_{u \in J} d^+(u, B) = \sum_{v\in B} d^-(v, J) \geq b(\ell-1)a'.
    \]
     Now, for $i \in [2k]$ let $I_i = [o+ (i - 2\ell - 1)a', o+(i-2\ell)a')$ and notice that $J = \bigcup_{i \in [2\ell]} I_i$.
     Therefore, by the pigeonhole principle there exists an $i$ such that
     \[
        \sum_{u \in I_i} d^+(u, B) \geq b\left(\frac{\ell-1}{2\ell}\right)a'.
     \]
    Fix such an $i$, and let $I = I_i$. Let $o'=j+(i-2\ell-1)a'$ be the left endpoint of $I$.
    Now the conclusion follows from \cref{lemma:dependent_random_choice} with $L_{\ref{lemma:dependent_random_choice}} = I$, $R_{\ref{lemma:dependent_random_choice}} = B$, $A_{\ref{lemma:dependent_random_choice}} = {\cal{F}}_{\ref{lemma:dependent_random_choice}} = \emptyset$.
\end{proof}

We are now ready to prove \cref{theorem:graded_upper_bound}.
\begin{proof}[Proof of \cref{theorem:graded_upper_bound}]
    Let $G$ be a graded digraph on $n$ vertices with a graded partition $V(G) = V_1 \cup \dots \cup V_H$ for some $H \in \mathbb{N}$ and let $\Delta^+$ and $\Delta^-$ be its maximum out- and in-degree, respectively.
    Moreover, for each $i \in [H-1]$ let $\Delta_i^-$ be the maximal in-degree in the induced subgraph $D[V_i \cup V_{i+1}]$ and set $\Delta^-_0 = \Delta^-_H = 0$. Note that we may assume that $\Delta_i^-\geq 1$ for all $i \in [h-1]$, for otherwise the underlying graph of $G$ is disconnected, and we obtain the desired bound on $\vv r(G)$ by summing up $\vv r(G')$ for every connected component $G'$ of $G$.

    We define $\varepsilon = 2/\Delta^-$ and $\ell = 4\Delta^- + 4$, and note that $\frac{2\ell}{\ell-1} \leq 2 + \varepsilon$.
    We define integers $c_s$, $c_b$, and $c_a$ by
    \[
        c_s = 32, \qquad c_b = 2000 \Delta^+ \Delta^- c_s, \qquad\text{and} \qquad c_a = 2c_b.
    \]
    We further define
    \[
        n_i = \sum_{j=i}^H \frac{(2 + \varepsilon)^{2 (\Delta_{j-1}^- + \Delta_j^-)} |V_j|}{2^{j-i}},
    \]
    and let
    \[
        a_i = c_a n_i \qquad \text{ and } \qquad b_i = c_b (2 + \varepsilon)^{-2\Delta_i^-}n_i.
    \]
    We will build sets $A_H,\dots,A_1$ such that $\ab{A_i}\geq b_i$ and that $A_i$ lies in an interval of length $a_i$. We now let
    \[
        s_i = c_s (2 + \varepsilon)^{-2(\Delta_i^- + \Delta_{i-1}^-)} n_i \qquad \text{ and } \qquad \delta_i =\frac{1}{4\Delta^- \Delta^+} \left(2^{-1/2}\frac{s_{i+1}}{b_{i+1}}\right)^{\Delta_i^-} .
    \]
    We will further guarantee that $(A_i, \emptyset)$ is an $(A_{i+1}, \Delta_i^-, s_{i+1}, \delta_i\binom{|A_{i+1}|}{\Delta_i^-})$-out-community. 
    We note for future reference that 
    \[
        n_i = (2+\varepsilon)^{2(\Delta_{i-1}^-+\Delta_i^-)}\ab{V_i} + \frac 12 n_{i+1}.
    \]
    In particular, this implies that $a_i \geq a_{i+1}/2$ and that $s_i \geq c_s \ab{V_i}= 32|V_i|$.
    
    Let $N = \sum_{i=1}^H 2 \ell a_i$, let $T$ be a tournament on $N$ vertices, and fix a median order $v_1, \dots, v_N$ of $T$.
    We now claim that we can find integers $o_1,\dots,o_H$ and disjoint sets $A_1,\dots,A_H \subseteq V(T)$ satisfying the following properties.
    \begin{itemize}
        \item $o_H = N - a_h + 1$ and $o_i - 2\ell a_{i-1} \leq  o_{i-1} \leq o_i - a_{i-1}$
        \item $A_i \subseteq [o_i, o_i + a_i)$,
        \item $|A_i| \geq b_i$, {and}
        \item for each $i \in [H-1]$, $(A_i, \emptyset)$ is an $(A_{i+1}, \Delta_{i}^-, s_{i+1}, \delta_i\binom{|A_{i}|}{\Delta_i^-})$-out-community.
    \end{itemize}
    Note that in particular, we will have that the sets $A_i$ are disjoint and that $o_1 \geq 1$.
    
    Setting $o_H = N - a_h + 1$ and $A_H = [o_H, o_H + a_H)$ clearly satisfies these properties, so suppose now that for some $i \in [H-1]$ we have defined $A_{i'}$ and $o_{i'}$ for all $i < i' \leq H$.
    By applying  \cref{lemma:graded_backward_step} with $o_{\ref{lemma:graded_backward_step}} = o_{i+1}$, $B_{\ref{lemma:graded_backward_step}} = A_{i+1}$, $a_{\ref{lemma:graded_backward_step}} = a_{i+1}, a'_{\ref{lemma:graded_backward_step}} = a_i \geq a_{i+1}/2, b = b_{i+1}, k_{\ref{lemma:graded_backward_step}} = 2\Delta^-_i, s = s_{i+1}, \ell_{\ref{lemma:graded_backward_step}}=\ell$ and $\Delta^-_{\ref{lemma:graded_backward_step}} = \Delta_i^-$ we can find $o_{i+1} - 2\ell a_i \leq  o_{i} \leq o_{i+1} - a_i$ and $A_i \subseteq [o_i, o_i + a_i)$ such that
    \begin{itemize}
        \item $|A_i| \geq \frac{a_i}{2}(2+ \varepsilon)^{-2\Delta^-_i} = b_i$, {and}
        \item $(A_i, \emptyset)$ is a $(A_{i+1}, \Delta_{i}^-, s_{i+1}, 4 \cdot (2+\varepsilon)^{2\Delta_i^-}\binom{a_i}{\Delta_i^-}(\frac{s_{i+1}}{b_{i+1}})^{2\Delta_i^-})$-out-community
    \end{itemize}
    We now note that, since $b_i \geq 2\Delta_i^-$, we have that
    \[
        \frac{\binom{a_i}{\Delta_i^-}}{\binom{b_i}{\Delta_i^-}}\leq 2^{\Delta_i^-} \left( \frac{a_i}{b_i} \right)^{\Delta_i^-} =  \left(2\cdot \frac{c_a(2+\varepsilon)^{2\Delta_i^-}}{c_b} \right)^{\Delta_i^-} = \left( 4(2+\varepsilon)^{2\Delta_i^-} \right)^{\Delta_i^-},
    \]
    where we plug in our definitions of $a_i,b_i,c_a$, and $c_b$. 
    Additionally, we have that
    \[
        \left( \frac{s_{i+1}}{b_{i+1}} \right)^{\Delta_i^-}= \left( \frac{c_s(2+\varepsilon)^{2\Delta_{i+1}^-}}{c_b(2+\varepsilon)^{2(\Delta_{i+1}^-+\Delta_{i}^-)} } \right)^{\Delta_i^-} = \left( \frac{c_s}{c_b (2+\varepsilon)^{2\Delta_{i}^-}} \right)^{\Delta_i^-} = \left(\frac{1}{2000\Delta^+\Delta^-(2+\varepsilon)^{2\Delta_i^-}}\right)^{\Delta_i^-}
    \]
    by our choices of $s_{i+1},b_{i+1}, c_s$, and $c_b$.
    Putting this all together, we see that
    \begin{align*}
        4 \cdot (2+\varepsilon)^{2\Delta_i^-}\binom{a_i}{\Delta_i^-}(\frac{s_{i+1}}{b_{i+1}})^{2\Delta_i^-}&= \left[4(2+\varepsilon)^{2\Delta_i^-} \left( \frac{s_{i+1}}{b_{i+1}} \right)^{\Delta_i^-}\right]\cdot \left[ \frac{\binom{a_i}{\Delta_i^-}}{\binom{b_i}{\Delta_i^-}}\left( \frac{s_{i+1}}{b_{i+1}} \right)^{\Delta_i^-} \right]\binom{b_i}{\Delta_i^-}\\
        &\leq \left[4(2+\varepsilon)^{2\Delta_i^-} \left( \frac{s_{i+1}}{b_{i+1}} \right)^{\Delta_i^-}\right] \left( \frac{1}{500\Delta^+\Delta^-} \right)^{\Delta_i^-}\binom{b_i}{\Delta_i^-}\\
        &\leq \left( 2^{-1/2}\frac{s_{i+1}}{b_{i+1}} \right)^{\Delta_i^-} \left( \frac{4\cdot 2^{1/2}(2+\varepsilon)^2}{500\Delta^+\Delta^-} \right)^{\Delta_i^-}\binom{b_i}{\Delta_i^-}\\
        &\leq \delta_i\binom{|A_i|}{\Delta_i^-},
    \end{align*}
    where the final step uses that $\varepsilon\leq 2$, that $\Delta_i^-\geq 1$, and our definition of $\delta_i$. This shows that the set $A_i$, as defined above, satisfies the desired properties. Continuing inductively in this way,
    we are able to find all the sets $A_1, \dots, A_H$.

    Having found these sets we now want to embed each $V_i$ into $A_i$, starting this time with $V_1$.
    For $i \in [h]$, let $S_i = \bigcup_{j =1}^i V_j$. For each $i \in [H]$ we find a function $\phi_i: S_i \to (A_1 \cup \dots \cup A_i)$ such that
    \begin{itemize}
        \item $\phi_i$ is an embedding of $G[S_i]$ into $T[A_1 \cup \dots \cup A_i]$,
        \item $\phi(V_j) \subseteq A_j$ for each $1 \leq j \leq i $, {and}
        \item for each $i \in [H-1]$ and $v \in V_{i+1}$, the vertices $\phi_i(N^-_G(v)) \subseteq A_i$ have at least $s_{i+1}$ common out-neighbors in $A_{i+1}$.
    \end{itemize}

    For all $i \in [H-1]$, we define ${\cal{F}}_i$ to be the set of all $S \subseteq A_i$ with $|N^+(S) \cap A_{i+1}| \leq s_{i+1}$ and notice that ${\cal{F}}$ is upward-closed and has $\Delta^-_i$-density at most $\delta_i$. 
    We can therefore find $\phi_1$ by applying \cref{lemma:forward_step} with $a_{\ref{lemma:forward_step}} = b_{\ref{lemma:forward_step}} = b_1$, $W_{1,\ref{lemma:forward_step}} = A_1$, $W_{2,\ref{lemma:forward_step}} = A_2$, $D_{\ref{lemma:forward_step}} = G[V_1 \cup V_2]$, ${\cal{F}}_v = {\cal{F}}_1$ for all $v \in V_2$ and $f(u) = A_1$ for all $u \in V_1$.
    Suppose now that for some $i \in [H-2]$ we have found $\phi_i$ satisfying the conditions above.
    For $u \in V_{i+1}$, let $f(u) = N_T^+(\phi_i(N_G^-(v))) \cap A_{i+1}$ and note that $|f(u)| \geq s_{i+1}$ for all $u \in V_{i+1}$.
    We can now again apply \cref{lemma:forward_step} with $W_{1,\ref{lemma:forward_step}} = A_{i+1}$, $W_{2,\ref{lemma:forward_step}} = A_{i+2}$, $a_{\ref{lemma:forward_step}} = b_{i+1}$, $b_{\ref{lemma:forward_step}} = s_{i+1}$, ${\cal{F}}_v = {\cal{F}}_{i+1}$ for all $v \in V_{i+2}$ and $D_{\ref{lemma:forward_step}} = G[V_{i+1} \cup V_{i+2}]$.
    We get a function $\phi$ such that
    \begin{itemize}
        \item for all $v \in V_{i+1}$ we have $\phi(v) \in N_T^+(\phi_i(N_G^-(v)))$, {and}
        \item for all $v \in V_{i+2}$ we have $|N_T^+(\phi(N_G^-(v)))| \geq s_{i+2}$.
    \end{itemize}
    Thus, the function
    \[
        \phi_{i+1}(v) \coloneqq \begin{cases}
            \phi(v), & v \in V_{i+1} \\
            \phi_i(v), & v \in S_i
        \end{cases}
    \]
    satisfies the conditions above.
    
    Proceeding in this way inductively, we can therefore find $\phi_{H-1}$ satisfying the same properties.
    Now, since for all $v \in V_H$ we have $|N_T^+(\phi_{H-1}(N_G^-(v)))| \geq s_h \geq |V_H|$ we can greedily extend $\phi_{H-1}$ into $\phi_h$ satisfying the above conditions. In particular, $\phi_H$ is an embedding of $G$ into $T$.
    
    To complete the proof, it remains to estimate $N$, the number of vertices in $T$.
    Plugging in our choice of $\varepsilon = 2/\Delta^-$, as well as the estimates $1+x\leq e^x$ and $\Delta_j^-\leq \Delta^-$ for all $j$, we find that
    \[
        (2+\varepsilon)^{2\Delta_{j-1}^- + 2\Delta_{j}^-} = (2(1+\varepsilon/2))^{{2\Delta_{j-1}^- + 2\Delta_{j}^-}} \leq  e^{4} 2^{2\Delta_{j-1}^- + 2\Delta_{j}^-}.
    \] 
    Additionally, we have that
    \begin{align*}
        \sum_{i=1}^H n_i &= \sum_{i=1}^H \sum_{j=i}^H \frac{(2 + \varepsilon)^{2 (\Delta_{j-1}^- + \Delta_j^-)} |V_j|}{2^{j-i}} \\
        &= 
        \sum_{j=1}^H (2 + \varepsilon)^{2 (\Delta_{j-1}^- + \Delta_j^-)} |V_j| \sum_{i=1}^j \frac 1{2^{j-i}}\\
        &\leq 2 \sum_{j=1}^H (2 + \varepsilon)^{2 (\Delta_{j-1}^- + \Delta_j^-)} |V_j|.
    \end{align*}
    Therefore,
    \[
        N = 2\ell \sum_{i=1}^H a_i = 2\ell c_a \sum_{i=1}^H n_i \leq 4\ell c_a \sum_{j=1}^H (2+\varepsilon)^{2\Delta_{j-1}^- + 2\Delta_{j}^-} |V_j| \leq c (\Delta^-)^2 \Delta^+ \sum_{j=1}^H 2^{2(\Delta_{j-1}^- + \Delta_{j}^-) } |V_j|,
    \]
    for the absolute constant $c=4\cdot 5\cdot 32 \cdot 2000 \cdot 2\cdot e^4 \leq 10^9$.
    Since $T$ was an arbitrary tournament on $N$ vertices, we have shown that 
    \[
        \vv{r}(G) \leq 10^9 (\Delta^-)^2 \Delta^+ \sum_{i=1}^H 2^{2(\Delta_{i-1}^- + \Delta_{i}^-) } |V_i|.\qedhere
    \]
\end{proof}

\section{Proofs of the technical lemmas from Section \ref{section:GRR_lemmas}}\label{appendix:GRR}
In this section, we prove the two lemmas from \cref{section:GRR_lemmas}. We begin with \cref{lemma:guest_graph}.
\begin{proof}[Proof of \cref{lemma:guest_graph}]
    We take any $c_0,c_1$ satisfying $1 < c_1^2 < c_0 < (5/4)^{1/202}$ and choose $\Delta_0$ so that $(c_1^2/c_0)^{\Delta_0} < 0.1$, $((0.8)^{1/101}c_0^2)^{\Delta_0} < 1/16$ and $(1 - 10^4)^{\Delta_0/101} < 1/8$. Note that we can choose such a $\Delta_0$ since all three of these inequalities are satisfied for sufficiently large $\Delta_0$. Let moreover $\Delta \geq \Delta_0$, $d = \Delta/101$, and $m = 1.01n$.

    To obtain our graph $H$, we will first draw a bipartite graph $G$ uniformly at random from the set of all bipartite graphs with $dm$ edges and with vertex classes $V'$ and $V''$ of size $m$ each. 
    Then we will remove the $n/100$ largest degree vertices on each side to obtain $H$.
    Since the number of vertices of degree larger than $\Delta$ in $D$ is at most $\frac{dm}{\Delta + 1} < \frac{m}{101}=\frac n{100}$, the maximum degree of $H$ is at most $\Delta$ with probability $1$.
    It thus suffices to show that $H$ will also satisfy the other two properties with positive probability.

    For the first one, let us bound the probability that there exist partitions $V' = V_1' \cup \dots \cup V_k' \cup D_X \cup D'$ and $V'' = V_1'' \cup \dots \cup V_k'' \cup D_Y \cup D''$ with $|D'| = |D''| = n/100$,  $|D_X|, |D_Y| \leq n/50$, and $|V_i'|, |V_i''| \leq (c_1/c_0)^\Delta n$ for all $i \in [k]$, such that
    \[
        \sum_{i \neq j: e_H(V_i',V_j'') > 0} |V_i'||V_j''| \leq 0.55(0.98n)^2.
    \]
    To do that, notice that since 
    \[
        \sum_{i=1}^k |V_i'||V_i''| \leq k \left(\left(\frac{c_1}{c_0}\right)^{\Delta}n\right)^2 = \left(\frac{c_1^2}{c_0}\right)^\Delta n^2 \leq \left(\frac{c_1}  {c_0}\right)^{\Delta_0} n^2 < 0.1n^2,
    \]
    such a partition must satisfy
    \[
        \sum_{i \neq j: e_G(V_i', V_j'') = 0} |V_i'||V_j''| \geq (0.98n)^2 - 0.1n^2 - 0.55(0.98n)^2 \geq 0.2m^2.
    \]
    Therefore, by the union bound, the probability that such a partition exists is at most
    \[
        (k+2)^{2m} 2^{k^2} \frac{\binom{0.8m^2}{dm}}{\binom{m^2}{dm}} < (2k)^{2m} 2^{k^2} (0.8)^{dm} < 8^m ((0.8)^{1/101}c_0^2)^{\Delta_0 m} < \frac 12,
    \]
    where $(k+2)^{2m}$ is a bound on the number of partitions, $2^{k^2}$ bounds the number of possible choices of pairs $(V_i', V_j'')$ with no edges in between them and ${\binom{0.8m^2}{dm}}/{\binom{m^2}{dm}}$ is a bound on the probability that indeed no edges fall between them.

    Similarly, the probability that there exist $X' \subseteq X$ and $Y' \subseteq Y$ of sizes at least $0.01n$ each such that $e_H(X', Y') = 0$ is at most
    \[
        2^{2m} \frac{\binom{(1 - 10^4)m^2}{dm}}{\binom{m^2}{dm}} < 2^{2m} (1 - 10^4)^{dm} \leq 2^{2m} (1 - 10^4)^{\Delta_0 m / 101} < \frac 12,
    \]
    where $2^{2m}$ bounds the number of choices of $X'$ and $Y'$ and the fraction bounds the probability that there are no edges between them.
    Thus, by the union bound, there exists a choice of $G$ such that both properties are satisfied, implying the existence of the desired $H$. 
\end{proof}
We now turn to the proof of \cref{lemma:host_graph}.
\begin{proof}[Proof of \cref{lemma:host_graph}]
    Note that $2x = \sum_{i=1}^k (f(i)+g(i)) \leq k$, hence the statement is vacuous if $2x>k$, as in this case there exist no such functions $f,g$. Thus, we assume henceforth that $2x \leq k$, and in particular that $k > 10^8 \log 2$. We let $R$ be a uniformly random tournament with vertex set $[k]$.

    We first claim that we can assume that there exist $i_0 \neq j_0$ such that for all $i \neq i_0$ and all $j \neq j_0$ we have
    \[
        f(i), g(j) \in \{0,1\}.
    \]
    Indeed, for any fixed outcome of $R$, suppose that $f$ and $g$ maximize $W$ and that there exists an $i$ such that $0 < f(i), g(i) < 1$.
    Now consider the sums $W_f(i) \coloneqq \sum_{j:ij \in E(R)} g(j)$ and $W_g(i) \coloneqq \sum_{j: ji \in E(R)} f(j)$.
    If $W_f(i) \geq W_g(i)$ then define new functions $f'$ and $g'$ such that there are equal to $f$ and $g$ except that $f'(i) = f(i) + g(i)$ and $g'(i) = 0$.
    Otherwise, we set $f'(i) = 0$ and $g'(i) = f(i) + g(i)$.
    In either case, we have $W' \geq W$ for the corresponding quantity $W'$.

    Thus, we can assume that $\min \{ f(i), g(i)  \} = 0$ for all $i$. 
    Now suppose that there exist $i \neq j$ such that $0 < g(i), g(j) < 1$.
    Again in case $W_g(i) \geq W_g(j)$, we let $\varepsilon_{ij} = \min \{ g(j), 1 - g(i) \}$ and let $g'(i) = g(i) + \varepsilon_{ij}$ and $g'(j) = g(j) - \varepsilon_{ij}$.
    Otherwise, we do the same with $i$ and $j$ swapped and in both cases we get $W' \geq W$.
    By the same argument, we can also assume that for at most one $i_0$ we have $0 < f(i_0) < 1$.

    Now, assuming $f$ and $g$ satisfy the property above, define $T = \{i: f(i) = 1\}$ and $S = \{ j: g(j) = 1\}$ and let $t = |T|, s = |S|$.
    By our assumptions, we have that $t+s \leq 2x < t + s + 2$ and $2x \leq k$.
    Additionally, we have
    \[
    W = \sum_{ij \in E(R)} f(i)g(j) = e_R(T, S) + g(j_0)W_g(j_0) + f(i_0)W_f(i_0) \leq e_R(T, S) + 2x.
    \]
    In particular, if $W > 0.51x^2$ we find that $$e_R(T,S) > 0.51x^2 - 2x \geq 0.501x^2 \geq 0.501\frac{(s+t)^2}{4}.$$
    We now claim that with positive probability (over the randomness in $R$), there exist no sets $S,T \subseteq V(R)$ satisfying this inequality.
    Suppose first that $t \leq s$.
    Note that we must have $t > s/7$ since otherwise $e_R(T, S)  \leq ts < 0.5(t+s)^2 / 4$.
    Similarly, we must have $s > s_0 \coloneqq 2 \cdot 10^7 \log k$.
    For any fixed disjoint $T,S$ we have that $e_R(T, S) \sim \operatorname{Bin}(ts, \frac 12)$ and thus by Chernoff's inequality, and using $(t+s)^2/4 \geq ts \geq s^2/7$, we get
    \[
        \Pr[e_R(T, S) > 0.501 (t+s)^2 /4] \leq \Pr[e_R(T, S) > 0.501ts] < e^{-10^{-7}s^2}.
    \]
    Moreover, for $s>s_0$, we have
    \[
        e^{-10^{-7}s^2} < e^{-10^{-7}s\cdot s_0}=k^{-2s}.
    \]
    Therefore, the probability that such $T$ and $S$ exist is at most
    \[
        \sum_{s=s_0}^{k} \sum_{t=s/7}^{s} \binom{k}{s}\binom{k}{t}e^{-10^{-7}s^2} \leq \sum_{s=s_0}^{k} \sum_{t=s/7}^{s} \left(\frac{ek}s\right)^{2s}k^{-2s} \leq k \sum_{s=s_0}^{k}\left(\frac{e^2}{s^2}\right)^s < k^2\left(\frac{e^2}{s_0^2}\right)^{s_0} < \frac 12,
    \]
    where in the first inequality we use that since $t \leq s$, we have that $\binom{k}{s} \binom{k}{t} \leq \binom{k}{s}^2 < (ek/s)^{2s}$. By interchanging the roles of $s$ and $t$, we obtain the same bound in case $t \geq s$. Thus,  we find that $R$ satisfies the desired property with positive probability.
\end{proof}

\end{document}